\documentclass[preprint,11pt]{amsart}

\usepackage[margin=0.9in]{geometry}

\usepackage{amssymb}
\usepackage{amsmath}
\usepackage{graphicx}

\usepackage[english]{babel}
\usepackage[utf8]{inputenc}
\usepackage[T1]{fontenc}
\usepackage{amsthm}
\usepackage{caption}

\usepackage{algorithmic}
\usepackage{algorithm}

\usepackage[all,cmtip]{xy}

\usepackage{lipsum}

\usepackage{tikz}
\usepackage{mathrsfs}
\usetikzlibrary{automata,positioning,shapes,arrows}

 \newtheorem{thm}{Theorem}[section]
 \newtheorem{cor}[thm]{Corollary}
 \newtheorem{lem}[thm]{Lemma}
 
 \newtheorem{defn}[thm]{Definition}
 \newtheorem{rem}[thm]{Remark}
 \newtheorem{ex}{Example}
  \newtheorem{pr}{Problem}

\begin{document}

\title{Sparse System Identification by Low-Rank Approximation}

\author{Fredy Vides} 

\address{Scientific Computing Innovation Center, School of Mathematics and Computer Science, Universidad Nacional Aut\'onoma de Honduras, Tegucigalpa, Honduras}

\email{fredy.vides@unah.edu.hn}

\date{\today}

\begin{abstract}
In this document, some general results in approximation theory and matrix analysis with applications to sparse identification of time series models and nonlinear discrete-time dynamical systems are presented. The aforementioned theoretical methods are translated into algorithms that can be used for sparse model identification of discrete-time dynamical systems, based on structured data measured from the systems. The approximation of the state-transition operators that are determined primarily by matrices of parameters to be identified based on data measured from a given system, is approached by identifying conditions for the existence of low-rank approximations of submatrices of the trajectory matrices corresponding to the measured data, that can be used to compute approximate sparse representations of the matrices of parameters. One of the main advantages of the low-rank approximation approach presented in this document, concerns the parameter estimation for linear and nonlinear models where numerical or measurement noise could affect the estimates significantly. Prototypical algorithms based on the aforementioned techniques together with some applications to approximate identification and predictive simulation of time series models with symmetries and nonlinear structured dynamical systems in theoretical physics, fluid dynamics and weather forecasting are presented.
\end{abstract}






\maketitle

\section{Introduction}

In this document, some structured matrix approximation problems that arise in the fields of system identification and model order reduction of large-scale structured dynamical systems are studied. 

The main purpose of this document is to present some theoretical and computational techniques that have been developed for the approximate structure preserving sparse identification of discrete-time dynamical systems based on structured data measured from the systems. In particular, we explore the idea of using low-rank approximations of submatrices of the Hankel-type trajectory matrices corresponding to the data samples, for the computation of the approximate sparse representations of the matrices of parameters to be identified as part of the model identification processes considered in this document. 

As part of the process previously described, some general results in approximation theory and matrix analysis with applications to sparse identification of time series models and nonlinear dynamical systems are obtained. The approximation of the corresponding state-transition operators determined by matrices of parameters to be identified, is approached by identifying conditions for the existence of easily computable integers that can be applied to estimate the computability of approximate sparse representations of the matrices of parameters, and as a by-product of the computation of these numbers one obtains low-rank approximations of submatrices of the trajectory matrices corresponding to some data measured from the system under study, that can be used to compute the sparse approximants of the matrices of parameters.

One of the main advantages of the low-rank approximation approach presented in this document, concerns the parameter estimation for linear and nonlinear models where numerical or measurement noise could affect the estimates significantly. Another advantage is the reduction of arithmetic complexities obtained as a consequence of the application of low-rank approximation techniques, as studied by Chen, Avron and Sindhwani in \cite{JMLR:v18:15-376} in the context of scalable nonparametric learning. The low-rank approximation approach implemented in this study makes sparse linear least squares solver algorithms like algorithm \ref{alg:main_SLMESolver_alg_1}, suitable for parallelization.
 
The identification and predictive numerical simulation of the evolution laws for discrete-time systems are highly important in predictive data analytics, for models related to the automatic control of systems and processes in science and engineering in the sense of \cite{BROCKETT20081,finite_quantum_control_systems,DMD_Kutz}. Part of the motivation for the development of the techniques presented in this paper came from matrix approximation problems that arise in the fields of system identification in the sense of \cite{finite_state_machine_approximation,APSSA,ChaosEqKoopman}, and the computation of digital twins as considered in \cite{DTwin1,DTwin2}. 

The study reported in this document was inspired by the theoretical and computational questions and results presented by Salova, Emenheiser, Rupe, Crutchfield, and D’Souza in \cite{ChaosEqKoopman}, by Boutsidis and Magdon-Ismail in \cite{SpLSRegression}, by Finzi, Stanton, Izmailov and Wilson in \cite{CNNEquivariance}, by Brockett and Willsky in \cite{finite_state_systems}, by Moskvina and Schmidt in \cite{APSSA}, by Shmid in \cite{DMD_Schmid}, by Proctor, Brunton and Kutz in \cite{DMD_Kutz}, by Kaiser, Kutz and Brunton in \cite{Sparse_ID_Kutz}, by Kaheman, Kutz and Brunton in \cite{SindyPIPredictive}, by Freedman and Press in \cite{FREEDMAN19961}, by Franke and Selgrade in \cite{FRANKE200364}, by Farhood and Dullerud in \cite{FARHOOD2002417}, by Schaeffer, Tran, Ward and Zhang in \cite{StructuredDouglasRachfordID}, and by Loring and Vides in \cite{SymmLogs}. 

Among the previous references, from a computational perspective, two key sources of inspiration for the work reported in this document were the amazing computational implementations of {\bf SINDy} and Douglas-Rachford algorithms for sparse nonlinear system identification along the lines of \cite{BruntonSINDy}, \cite{SindyPIPredictive} and \cite{StructuredDouglasRachfordID}. One of the objectives of the work reported in this article is to investigate the effect that the use of low-rank matrix approximation techniques to preprocess the data used as part of the model identification process would have on the performance of sparse model identification programs built over the basis of ideas used by Brunton, Kaheman, Kutz and Proctor in \cite{BruntonSINDy} and \cite{SindyPIPredictive}.

The main contribution of this article is the use of low-rank matrix approximation techniques to produce fast and easy to use sparse linear least squares solver algorithms, that can be effectively applied to sparse model identification processes. The low-rank approximation techniques that have been implemented provide a way to control the predictive model sensitivity to noise in the training data. As part of this research project, several general purpose computational tools for sparse model identification in science and engineering have been developed. 

The constructive nature of the results presented in the sections \S\ref{section:linear-solvers} and \S\ref{section_SDSI} of this document allows one to derive prototypical algorithms like the ones presented in \S\ref{section_algorithms}. Some numerical implementations of this prototypical algorithms based on Matlab, Pyhton, Julia and Netgen/NGSolve are presented in \S\ref{section_numerics}.

\section{Preliminaries and Notation}
\label{notation}

In this study, every time we refer to a system we will be considering a discrete-time dynamical system that can be described as a pair $(\Sigma,\mathscr{T})$ determined by a set of {\em states} $\Sigma\subset \mathbb{C}^n$, and a function $\mathscr{T}:\Sigma\to \Sigma$ that will called a transition operator, such that for any time series determined by a sequence $\{x_t\}_{t\geq 1}\subset \Sigma$, we will have that $\mathscr{T}(x_t)=x_{t+1}$. For systems whose state spaces are contained in $\mathbb{R}^n$ we will consider the usual identification of $\mathbb{R}$ with the real line in $\mathbb{C}$, and will apply the system identification techniques presented in this study, considering suitable restrictions when necessary.

We will write $\mathbb{Z}^+$ to denote the set of positive integers $\mathbb{Z}\cap [1,\infty)$. 

Given a set $S$, we will write $\sharp(S)$ to denote the number of elements in $S$.

In this document the symbol $\mathbb{C}^{n\times m}$ will denote the algebra of $n\times m$ complex matrices, and we will write $I_n$ to denote the identity matrix in $\mathbb{C}^{n\times n}$ and $\mathbf{0}_{n,m}$ to denote the zero matrix in $\mathbb{C}^{n\times m}$, when $m=n$ we will write $\mathbf{0}_n$ instead of $\mathbf{0}_{n,n}$. From here on, given a matrix $X=[X_{ij}]\in \mathbb{C}^{m\times n}$, we will write $X^\ast$ to denote the conjugate transpose of $X$ determined by $X^\ast=\overline{X^\top}=[\overline{X}_{ji}]$ in $\mathbb{C}^{n\times m}$. We will represent vectors in $\mathbb{C}^n$ as column matrices in $\mathbb{C}^{n\times 1}$ and as $n$-tuples.

Given any matrix $A\in \mathbb{C}^{m\times n}$ we will write $\mathrm{rk}(A)$ to denote the rank of $A$, that corresponds to the maximal number of linearly independent columns of $A$.

Given $x\in\mathbb{C}^n$ we will write $\|x\|$ to denote the $\ell_2$-norm in $\mathbb{C}^n$ determined by $\|x\|=\sqrt{x^\ast x}=(\sum_{j=1}^n |x_j|^2)^{1/2}$, and we will write $\|x\|_\infty$ to denote the $\ell_\infty$-norm in $\mathbb{C}^n$ determined by the expression $\|x\|_\infty=\max_{1\leq j\leq n}|x_j|$ for each $x\in \mathbb{C}^n$. 

In this document we will write $\hat{e}_{j,n}$ to denote the matrices in $\mathbb{C}^{n\times 1}$ representing the canonical basis of $\mathbb{C}^{n}$ (each $\hat{e}_{j,n}$ is the $j$-column of $I_n$), that are determined by the expressions
\begin{equation}
\hat{e}_{j,n}=
\begin{bmatrix}
\delta_{1,j} & \delta_{2,j} & \cdots & \delta_{n-1,j}  & \delta_{n,j}
\end{bmatrix}^\top
\label{ej_def}
\end{equation}
for each $1\leq j\leq n$, where $\delta_{k,j}$ is the Kronecker delta determined by the expression.
\begin{equation}
\delta_{k,j}=
\left\{
\begin{array}{l}
1, \:\: k=j\\
0, \:\: k\neq j
\end{array}
\right.
\label{delta_def}
\end{equation}

A matrix $P\in \mathbb{C}^{n\times n}$ will be called an orthogonal projector whenever $P^2=P=P^\ast$. A matrix $Q\in \mathbb{C}^{m\times n}$ such that the matrices $QQ^\ast$ and $Q^\ast Q$ are orthogonal projectors will be called a partial isometry. We will write $\mathbb{U}(n)$ to denote the group of unitary matrices in $\mathbb{C}^{n\times n}$ defined by the expression $\mathbb{U}(n)=\{X\in \mathbb{C}^{n\times n}:X^\ast X=XX^\ast=I_n\}$.

Given $X=[x_{jk}]\in \mathbb{C}^{m\times n}$ and $Y=[y_{pq}]\in \mathbb{C}^{r\times s}$, we will write $X\otimes Y$ to denote the Kronecker product defined by the expression $X\otimes Y=[x_{jk}Y]\in \mathbb{C}^{(mr)\times (ns)}$.

Given $X\in \mathbb{C}^{n\times m}$ we will write $\|X\|_F$ to denote the Frobenuius norm of $X$ defined by
\begin{equation}
\|X\|_F=\sqrt{\mathrm{tr}(X^\ast X)}
\label{eq:frob-norm}
\end{equation}
where $\mathrm{tr}$ denotes the trace of a matrix, defined for any $A=[a_{jk}]\in \mathbb{C}^{n\times n}$ by the expression.
\[
\mathrm{tr}(A)=\sum_{j=1}^n a_{jj}
\]

Given a finite set of vectors $\Sigma_T=\{x_1,\ldots,x_T\}\subset \mathbb{C}^n$ we will write $\mathcal{H}_L(\Sigma_T)$ to denote the Hankel-type trajectory matrix in $\mathbb{C}^{nL\times (T-L+1)}$ defined by the following expression.
\[
\mathcal{H}_L(\Sigma_T)=\begin{bmatrix}
x_1 & x_2 & x_3 & \cdots & x_{T-L+1}\\
x_2 & x_3 & x_4 & \cdots & x_{T-L+2}\\
x_3 & x_4 & x_5 & \cdots & x_{T-L+3}\\
\vdots & \vdots & \vdots & \cdots & \vdots\\
x_L & x_{L+1} & x_{L+2} & \cdots & x_{T}\\
\end{bmatrix}
\]

Given $a\in \mathbb{R}$, we will write $H_a$ to denote the function $H_a:\mathbb{R} \to \mathbb{R}$ defined by the following expression.
\begin{equation}
H_a(x)=\left\{
\begin{array}{ll}
1, & x> a\\
0, & x\leq a
\end{array}
\right.
\label{eq:H_delta_def}
\end{equation}

For any additional details on the notions previously considered, the reader is kindly referred to \cite{HornBook}, \cite{bhatia97} and \cite{GoVa13}.

\section{low-rank approximation and sparse linear least squares solvers}

\label{section:linear-solvers}

In this section some low-rank approximation methods with applications to the solution of sparse linear least squares problems are presented.

\begin{defn}
Given $\delta>0$ and a matrix $A\in \mathbb{C}^{m\times n}$, we will write $\mathrm{rk}_{\delta}(A)$ to denote the nonnegative integer determined by the expression
\[
\mathrm{rk}_{\delta}(A)=\sum_{j=1}^{\min\{m,n\}} H_\delta(s_j(A)),
\]
where the numbers $s_j(A)$ represent the singular values corresponding to an economy-sized singular value decomposition of the matrix $A$.
\end{defn}

\begin{lem}\label{lem:rk-delta-inv}
We will have that $\mathrm{rk}_\delta\left(A^\top\right)=\mathrm{rk}_\delta(A)$ for each $\delta>0$ and each $A\in \mathbb{C}^{m\times n}$.
\end{lem}
\begin{proof}
Given an economy-sized singular value decomposition 
\begin{align*}
U\begin{bmatrix}
s_1(A) & & &\\
& s_2(A) & &\\
&&\ddots&\\
&&& s_{\min\{m,n\}}(A)
\end{bmatrix}V=A
\end{align*}
we will have that
\begin{align*}
V^\top\begin{bmatrix}
s_1(A) & & &\\
& s_2(A) & &\\
&&\ddots&\\
&&& s_{\min\{m,n\}}(A)
\end{bmatrix}U^\top=A^\top
\end{align*}
is an economy-sized singular value decomposition of $A^\top$. This implies that
\begin{align*}
\mathrm{rk}_{\delta}\left(A^\top\right)=\sum_{j=1}^{\min\{m,n\}} H_\delta(s_j(A))=\mathrm{rk}_{\delta}(A)
\end{align*}
and this completes the proof.
\end{proof}

\begin{lem}\label{lem:rk-lem}
Given $\delta>0$ and $A\in \mathbb{C}^{m\times n}$ we will have that $\mathrm{rk}_\delta(A)\leq \mathrm{rk}(A)$.
\begin{proof}
We will have that $\mathrm{rk}(A)=\sum_{j=1}^{\min\{m,n\}} H_0(s_j(A))\geq \sum_{j=1}^{\min\{m,n\}} H_\delta(s_j(A))=\mathrm{rk}_\delta(A)$. This completes the proof.
\end{proof}
\end{lem}

\begin{thm}\label{thm:main-SpLSSolver}
Given $\delta>0$ and $y,x_1,\ldots,x_m\in \mathbb{C}^n$, let 
\begin{align*}
X&=\begin{bmatrix}
| & | &  & |\\
x_1 & x_2 & \cdots & x_{m}\\
| & | &  & |
\end{bmatrix}.
\end{align*}
If $\mathrm{rk}_\delta\left(X\right)>0$ and if we set $r=\mathrm{rk}_\delta\left(X\right)$ and $s_{n,m}(r)=\sqrt{r(\min\{m,n\}-r)}$ 
then, there are a rank $r$ orthogonal projector $Q$, $r$ vectors $x_{j_1},\ldots,x_{j_r}\in \{x_1,\ldots,x_m\}$ and $r$ scalars $c_1\ldots,c_r\in \mathbb{C}$ such that $\|X-QX\|_F\leq (s_{n,m}(r)/\sqrt{r})\delta$, and $\|y-\sum_{k=1}^r c_kx_{j_k}\|\leq \left(\sum_{k=1}^r|c_k|^2\right)^{\frac{1}{2}}s_{n,m}(r)\delta+\|(I_n-Q)y\|$.
\end{thm}
\begin{proof}
Let us consider an economy-sized singular value decomposition $USV=A$. If $u_j$ denotes the $j$-column of $U$, let $Q$ be the rank $r=\mathrm{rk}_\delta(A)$ orthogonal projector determined by the expression $Q=\sum_{j=1}^r u_ju_j^\ast$. It can be seen that
\begin{align*}
\|X-QX\|_F^2&=\sum_{j=r+1}^{\min\{m,n\}} s_j(X)^2\\
&\leq (\min\{m,n\}-r)\delta^2=\frac{s_{n,m}(r)^2}{r}\delta^2.
\end{align*}
Consequently, $\|X-QX\|_F\leq \frac{s_{n,m}(r)}{\sqrt{r}}\delta$.

Let us set.
\begin{align*}
\hat{X}&=
\begin{bmatrix}
| & | &  & |\\
\hat{x}_1 & \hat{x}_2 & \cdots & \hat{x}_{m}\\
| & | &  & |
\end{bmatrix}=QX\\
\hat{X}_y&=
\begin{bmatrix}
| & | &  & | & |\\
\hat{x}_1 & \hat{x}_2 & \cdots & \hat{x}_{m} & \hat{y}\\
| & | &  & | & |
\end{bmatrix}=Q\begin{bmatrix}
X & y\\
\end{bmatrix}
\end{align*}
Since by lemma \ref{lem:rk-lem} $\mathrm{rk}(X)\geq \mathrm{rk}_\delta(X)$, we will have that $\mathrm{rk}(\hat{X})=r=\mathrm{rk}_\delta(X)>0$, and since we also have that $\hat{x}_1,\ldots,\hat{x}_m,\hat{y}\in \mathrm{span}(\{u_1,\ldots,u_r\})$, there are $r$ linearly independent $\hat{x}_{j_1},\ldots,\hat{x}_{j_r}\in \{\hat{x}_1,\ldots,\hat{x}_m\}$ such that $\mathrm{span}(\{u_1,\ldots,u_r\})=\mathrm{span}(\{\hat{x}_{j_1},\ldots,\hat{x}_{j_r}\})$, this in turn implies that $\hat{y}\in \mathrm{span}(\{\hat{x}_{j_1},\ldots,\hat{x}_{j_r}\})$ and there are $c_1,\ldots,c_r\in \mathbb{C}$ such that $\hat{y}=\sum_{k=1}^r c_k\hat{x}_{j_k}$. It can be seen that for each $z\in \{x_1,\ldots,x_m\}$
\begin{align*}
\|z-Qz\|&\leq\|X-QX\|_F\leq \frac{s_{n,m}(r)}{\sqrt{r}}\delta,
\end{align*} 
and this in turn implies that
\begin{align*}
\left\|y-\sum_{k=1}^r c_kx_{j_k}\right\|&=\left\|y-\sum_{k=1}^r c_k x_{j_k}-\left(\hat{y}-\sum_{k=1}^r c_k \hat{x}_{j_k}\right)\right\|\\
&=\left\|y-\sum_{k=1}^r c_kx_{j_k}-Q\left(y-\sum_{k=1}^r c_k x_{j_k}\right)\right\|\\
&\leq \left(\sum_{k=1}^r|c_k|^2\right)^{\frac{1}{2}}s_{n,m}(r)\delta+\|(I_n-Q)y\|.
\end{align*}
This completes the proof.
\end{proof}

As a direct implication of theorem \ref{thm:main-SpLSSolver} one can obtain the following corollary.

\begin{cor}\label{cor:SpLSSolver}
Given $\delta>0$, $A\in \mathbb{C}^{m\times n}$ and $y\in \mathbb{C}^m$. If $\mathrm{rk}_\delta\left(A\right)>0$ and if we set $r=\mathrm{rk}_\delta\left(A\right)$ and $s_{n,m}(r)=\sqrt{r(\min\{m,n\}-r)}$ then, there are $x\in\mathbb{C}^n$ and a rank $r$ orthogonal projector $Q$ that does not depend on $y$, such that $\|Ax-y\|\leq \|x\|s_{n,m}(r)\delta+\|(I_m-Q)y\|$ and $x$ has at most $r$ nonzero entries.
\end{cor}
\begin{proof}
Let us set $x=\mathbf{0}_{n, 1}$ and $a_j=A\hat{e}_{j,n}$ for $j=1,\ldots,n$. Since $r=\mathrm{rk}_\delta\left(A\right)>0$ and $s_{n,m}(r)=\sqrt{r(\min\{m,n\}-r)}$, by theorem \ref{thm:main-SpLSSolver} we will have that there is a rank $r$ orthogonal projector $Q$ such that $\|A-QA\|_F\leq (s_{n,m}(r)/\sqrt{r})\delta$, and without loss of generality $r$ vectors $a_{j_1},\ldots,a_{j_r}\in \{a_1,\ldots,a_n\}$ and $r$ scalars $c_1\ldots,c_r\in \mathbb{C}$ with $j_1\leq j_2\leq \cdots\leq j_r$ (reordering the indices $j_k$ if necessary), such that $\|y-\sum_{k=1}^r c_ka_{j_k}\|\leq \left(\sum_{k=1}^r|c_k|^2\right)^{\frac{1}{2}}s_{n,m}(r)\delta+\|(I_m-Q)y\|$. If we set $x_{j_k}=c_{k}$ for $k=1,\ldots,r$, we will have that $\|x\|=\left(\sum_{k=1}^r|c_k|^2\right)^{\frac{1}{2}}$ and $Ax=\sum_{k=1}^r x_{j_k}a_{j_k}=\sum_{k=1}^r c_ka_{j_k}$. Consequently, $\|Ax-y\|\leq \|x\|s_{n,m}(r)\delta+\|(I_m-Q)y\|$. This completes the proof.
\end{proof}

Given $\delta>0$, and two matrices $A\in \mathbb{C}^{m\times n}$ and $Y\in\mathbb{C}^{m\times p}$, we will write $AX\approx_\delta Y$ to represent the problem of finding $X\in \mathbb{C}^{n\times p}$, $\alpha,\beta\geq 0$ and an orthogonal projector $Q$ such that $\|AX-Y\|_F\leq \alpha\delta+\beta\|(I_m-Q)Y\|_F$. The matrix $X$ will be called a solution to the problem $AX\approx_\delta Y$.

\begin{thm}\label{thm:MEQ-Solver}
Given $\delta>0$, and two matrices $A\in \mathbb{C}^{m\times n}$ and $Y\in\mathbb{C}^{m\times p}$. If $\mathrm{rk}_\delta\left(A\right)>0$ and if we set $r=\mathrm{rk}_\delta\left(A\right)$ 
then, there is a solution $X$ to the problem $AX\approx_\delta Y$ with at most $rp$ nonzero entries.
\end{thm}
\begin{proof}
Since $r=\mathrm{rk}_\delta\left(A\right)>0$ we can apply corollary \ref{cor:SpLSSolver} to each subproblem $Ax_j\approx_\delta Y\hat{e}_{j,p}$, to obtain $p$ solutions $x_1,\ldots,x_p$ and a rank $r$ orthogonal projector $Q$ such that $\|Ax_j-Y\hat{e}_{j,p}\|\leq \|x_j\|s_{n,m}(r)\delta+\|(I_m-Q)Y\hat{e}_{j,p}\|$ for each $j=1,\ldots,p$ with $s_{n,m}(r)=\sqrt{r(\min\{m,n\}-r)}$, and each $x_j$ has at most $r$ nonzero entries. Consequently, if we set
\begin{align*}
X=\begin{bmatrix}
| & & |\\
x_1 & \cdots & x_p\\
| & & |
\end{bmatrix}
\end{align*}
we will have that
\begin{align*}
\|AX-Y\|_F^2&= \sum_{j=1}^p \|Ax_j-Y\hat{e}_{j,p}\|^2\\
&\leq \sum_{j=1}^p(\|x_j\|s_{n,m}(r)\delta+\|(I_m-Q)Y\hat{e}_{j,p}\|)^2\\
&\leq \sum_{j=1}^p(\|X\|_F s_{n,m}(r)\delta+\|(I_m-Q)Y\|_F)^2\\
&=p(\|X\|_F s_{n,m}(r)\delta+\|(I_m-Q)Y\|_F)^2
\end{align*}
and this in turn implies that if we set $\alpha=\sqrt{p}s_{n,m}(r)\|X\|_F$ and $\beta=\sqrt{p}$, then
\begin{align*}
\|AX-Y\|_F\leq \alpha \delta + \beta \|(I_m-Q)Y\|_F.
\end{align*}
Therefore, $X$ es a solution to $AX\approx_\delta Y$ with at most $rp$ nonzero entries. This completes the proof.
\end{proof}

Although the sparse linear least squares solver algorithms and theoretical results in this article build on similar principles to the ones considered by Boutsidis and Magdon-Ismail in \cite{SpLSRegression} and by Brunton, Proctor, and Kutz in \cite{BruntonSINDy}. One of the main differences of the approach considered in this study with the approach used in \cite{SpLSRegression}, is that given a least squares linear matrix equation $X=\arg\min_Y\|AY-B\|_F$, eventhough the rank $\mathrm{rk}_\delta(A)$ approximation $A_\delta=QA$ corresponding to the matrix of coefficients $A$, is computed in a generic way using the orthogonal projector $Q$ determined by theorem \ref{thm:main-SpLSSolver} and corollary \ref{cor:SpLSSolver} that in turn can be computed using a truncated economy-sized singular value decomposition of $A$ with approximation error $\varepsilon=\mathcal{O}(\delta)>0$, the selection process of each ordered set of the columns of $A_\delta$ corresponding to each column of the sparse approximate representation $\hat{X}$ of the reference least squares solution $X$, is not random but $B$-dependent. And the main difference between the approach implemented in this study and the one implemented in \cite{BruntonSINDy}, is that the sparse approximation process used here is based on reference solutions of least squares problems that involve submatrices of the low-rank approximation $A_\delta$ of the matrix of coefficients $A$ instead of submatrices of the original matrix $A$.

More specifically, if $A$ has $n$ columns, once the low-rank approximation $A_\delta$ of $A$ is computed along the lines of theorem \ref{thm:main-SpLSSolver}, corollary \ref{cor:SpLSSolver} and theorem \ref{thm:MEQ-Solver}, for each column $x_j$ of an initial reference least squares solution $X$ of the problem $X=\arg\min_Y\|A_\delta Y-B\|_F$, one can set a threshold $\varepsilon>0$, find  an integer $N_j(\varepsilon)\leq \mathrm{rk}_\delta(A)$ and compute a permutation $\sigma_j:\{1,\ldots,n\}\to\{1,\ldots,n\}$ based on the moduli $|x_{i,j}|$ of the entries of $x_j$ according to the following conditions 
\begin{align*}
&|x_{\sigma_j(1),j}|\geq|x_{\sigma_j(2),j}| \geq \cdots \geq |x_{\sigma_j(N_j(\varepsilon)),j}|>\varepsilon,\\
&\varepsilon \geq |x_{\sigma_j(N_j(\varepsilon)+1),j}|\geq \cdots \geq |x_{\sigma_j(n-1),j}| \geq |x_{\sigma_j(n),j}|.
\end{align*}

For each ordered subset $\hat{a}_{\sigma_j(1)},\ldots,\hat{a}_{\sigma_j(N_j(\varepsilon))}$ of columns of $A_\delta$, we can solve the problems
\begin{align*}
\hat{x}_j=\arg\min_y\left\|\begin{bmatrix}
| & & |\\
\hat{a}_{\sigma_j(1)} & \cdots & \hat{a}_{\sigma_j(N_j(\varepsilon))}\\
| & & |
\end{bmatrix}y-b_j\right\|
\end{align*}
for each $j=1,\ldots,p$.

If we define $p$ vectors $\tilde{x}_1,\ldots,\tilde{x}_p$ according to the following assignments
\begin{align*}
\tilde{x}_{\sigma_j(k),j}=
\left\{
\begin{array}{ll}
\hat{x}_{k,j}, & 1\leq k\leq N_j(\varepsilon),\\
0, & N_j(\varepsilon) <k\leq n
\end{array}
\right.
\end{align*}
where $\tilde{x}_{i,j}$ and $\hat{x}_{k,j}$ denote the entries of each pair of vectors $\tilde{x}_j$ and $\hat{x}_j$, respectively, we obtain a new approximate sparse solution 
\begin{align*}
\tilde{X}=\begin{bmatrix}
| & & |\\
\tilde{x}_1 & \cdots & \tilde{x}_{p}\\
| & & |
\end{bmatrix}
\end{align*}
to the problem $\hat{X}=\arg\min_Y\|A_\delta Y-B\|_F$. Using $\tilde{X}$ as a new reference solution one can repeat this process for some prescribed number of times, or until some stopping criterion is met, in order to find a sparser representation of the initial reference solution $X$.

The results and ideas presented in this section can be translated into a sparse linear least squares solver algorithm described by algorithm \ref{alg:main_SLMESolver_alg_1} in \S \ref{section_algorithms}.

\section{Low-rank approximation methods for sparse model identification}

\label{section_SDSI}

\subsection{Sparse Identification of Transition Operators for Time Series with Symmetries}

Given a sequence $\{x_t\}_{t\geq 1}\subset \mathbb{C}^n$, we say that $\{x_t\}_{t\geq 1}$ is a time series of a system $(\Sigma,\mathscr{T})$, if $\{x_t\}_{t\geq 1}\subset \Sigma$ and $\mathscr{T}(x_t)=x_{t+1}$ for each $t\geq 1$. If in addition, there is a finite group $G_N=\{g_1,\ldots,g_{N}\}\subset \mathbb{U}(n)$ such that 
\begin{align*}
\mathscr{T}(g_jx_t)=g_j\mathscr{T}(x_{t}),
\end{align*}
for each $t\geq 1$ and each $g_j\in G_N$, we will say that the system $(\Sigma,\mathscr{T})$ is $G_N$-equivariant and that the sequence $\{x_t\}_{t\geq 1}$ is a time series with symmetries.

We will say that a matrix $S\in \mathbb{C}^{n\times n}$ is symmetric with respect to a finite group $G_N=\{g_1,\ldots,g_N\}\subset \mathbb{U}(n)$ if 
\[g_j S=S g_j\] 
for each $g_j\in G_N$. We will write $\mathbb{S}(n)^{G_N}$ to denote the set of all matrices in $\mathbb{C}^{n\times n}$ that are symmetric with respect to the group $G_N$. 

Given an integer $L\geq 1$, a finite group $G_N\subset \mathbb{U}(n)$ with $\sharp(G_N)=N$, and a sample $\Sigma_T=\{x_t\}_{t=1}^T$ from a time series $\{x_t\}_{t\geq 1}$ in the state space $\Sigma\subset \mathbb{C}^n$ of some $G_N$-equivariant system $(\Sigma,\mathscr{T})$, we will write $\mathscr{H}_{L}(\Sigma_T,G_N)$ to denote the structured block matrix with Hankel-type matrix blocks that is determined by the following expression.
\begin{align*}
\mathscr{H}_{L}(\Sigma_T,G_N)=\begin{bmatrix}
(I_L\otimes g_1)\mathscr{H}_L(\Sigma_T) & \cdots & (I_L\otimes g_N)\mathscr{H}_L(\Sigma_T)
\end{bmatrix}
\end{align*}

\begin{rem}\label{rem:Hankel-block}
Since $G_N\subset \mathbb{U}(n)$ is a group, one of the elements in $G_N$ is equal to $I_n$, consequently, one of the matrix blocks of $\mathscr{H}_{L}(\Sigma_T,G_N)$ is equal to $I_L\otimes I_n \mathscr{H}_{L}(\Sigma_T)=\mathscr{H}_{L}(\Sigma_T)$.
\end{rem}

Let us define the main sparse model identification problem for time series with symmetries. 

\begin{pr}{\bf Sparse model identification problem for time series with symmetries.}\label{prob:prob-1}
Given $\delta>0$, an integer $L>0$, a finite group $G_N\subset\mathbb{U}(n)$ with $\sharp(G_N)=N$, and a sample $\Sigma_T=\{x_t\}_{t=1}^{T}$ from a time series $\{x_t\}_{t\geq 1}\subset \mathbb{C}^n$ of a $G_N$-equivariant system $(\Sigma,\mathscr{T})$ with transition operator $\mathscr{T}$ to be identified. Let $\Sigma_{0}=\{x_1,\ldots,x_{T-1}\}$, $\Sigma_{1}=\{x_2,\ldots,x_{T}\}$, $\mathbf{H}_{L,k}=\mathscr{H}_{L}(\Sigma_k,G_N)$ for $k=0,1$ and $\tilde{G}_N=\{I_L\otimes g_j:g_j\in G_N\}$. Determine if it is possible to compute a sparse matrix $\hat{A}_T$, a matrix $A_T\in \mathbb{S}(nL)^{\tilde{G}_N}$, an orthogonal projector $Q$ and three nonnegative numbers $\mathbb{D},\mathbb{E}$ and $\mathbb{F}$ such that if we set
\begin{align*}
\nu&=\mathbb{D}\delta+\sqrt{nL}\|\mathbf{H}_{L,1}(I_{N(T-L)}-Q)\|_F,\\
\varepsilon&=\mathbb{E}\delta+\mathbb{F}\|\mathbf{H}_{L,1}(I_{N(T-L)}-Q)\|_F,
\end{align*}
then
\begin{align*}
&\|\mathbf{H}_{L,1}-\hat{A}_T\mathbf{H}_{L,0}\|_F\leq \nu,\\
&\left\|\mathscr{T}(x_t)-\mathscr{P}_L\hat{A}_T^tX_1\right\|\leq \varepsilon,\\
&\left\|\mathscr{T}(g_jx_t)-\mathscr{P}_L\hat{A}_T^t (I_L\otimes g_j) X_1\right\|\leq \varepsilon,\\
&\|\mathbf{H}_{L,1}-A_T\mathbf{H}_{L,0}\|_F\leq \nu,\\
&\left\|\mathscr{T}(x_t)-\mathscr{P}_LA_T^tX_1\right\|\leq \varepsilon,\\
&\left\|\mathscr{T}(g_jx_t)-\mathscr{P}_LA_T^t (I_L\otimes g_j) X_1\right\|\leq \varepsilon
\end{align*}
for each $t=1,\ldots,T-L$ and each $g_j\in G_N$, with $\mathscr{P}_L$ $=$ $\hat{e}_{1,L}^\top\otimes I_n$ and $X_1=\begin{bmatrix}
x_1 & \cdots & x_{L}
\end{bmatrix}^\top$. 
\end{pr}

\begin{defn}
We will write $\mathbf{SDSI}(\Sigma_T,G_N,L,\delta)$ to denote the set of $6$-tuples of solutions $(\hat{A}_T,A_T,Q$ $,$ $\mathbb{D},\mathbb{E},\mathbb{F})$ to problem \ref{prob:prob-1} based on data $\Sigma_T,G_N,L,\delta$.
\end{defn}

\begin{thm}\label{thm:main-thm}
Given $\delta>0$, a finite group $G_N\subset\mathbb{U}(n)$ with $\sharp(G_N)=N$, and a sample $\Sigma_T=\{x_t\}_{t=1}^{T}$ from a time series $\{x_t\}_{t\geq 1}\subset \mathbb{C}^n$ of a $G_N$-equivariant system $(\Sigma,\mathscr{T})$ with transition operator $\mathscr{T}$ to be identified. If there is an integer $L>0$ such that $\mathrm{rk}_\delta\left(\mathscr{H}_{L}\left(\{x_t\}_{t=1}^{T-1},G_N\right)\right)>0$ then, there is $(\hat{A}_T,A_T,Q,\mathbb{D},\mathbb{E},\mathbb{F})\in \mathbf{SDSI}(\Sigma_T,G_N,L,\delta)$.
\end{thm}
\begin{proof}
Let $\mathbf{H}_{L,k}=\mathscr{H}_{L}\left(\Sigma_k,G_N\right)$ for $k=0,1$, with each $\Sigma_{k}$ defined as in problem \ref{prob:prob-1}, and let us write $h_{j,k}$ to denote the $n\times (N(T-L))$ submatrix corresponding to the $j$-row block of $\mathbf{H}_{L,k}$ for $k=0,1$. By definition of $\mathbf{H}_{L,k}$ we will have that.
\begin{align}
h_{j,1}=h_{j+1,0}, \:\: 1\leq j\leq L-1
\label{eq:ID-cond-1}
\end{align}
Let $\Sigma_{T-1}=\{x_{t}\}_{t=1}^{T-1}$. It can be seen that $\mathscr{H}_{L}\left(\Sigma_{T-1},G_N\right)=\mathbf{H}_{L,0}$ and this in turn implies that
\begin{align}
\mathscr{H}_{L+1}\left(\Sigma_T,G_N\right)=\begin{bmatrix}
\mathscr{H}_{L}\left(\Sigma_{T-1},G_N\right) \\
h_{L,1}
\end{bmatrix}
=\begin{bmatrix}
\mathbf{H}_{L,0} \\
h_{L,1}
\end{bmatrix}.
\label{eq:embedding-ID-cond-2}
\end{align}
Since $\mathrm{rk}_\delta\left(\mathbf{H}_{L,0}\right)=\mathrm{rk}_\delta\left(\mathscr{H}_{L}\left(\Sigma_{T-1},G_N\right)\right)>0$, by lemma \ref{lem:rk-delta-inv} we will have that
\begin{align}
\mathrm{rk}_\delta\left(\mathbf{H}_{L,0}^\top\right)=\mathrm{rk}_\delta\left(\mathbf{H}_{L,0}\right)>0.
\label{eq:ID-cond-2}
\end{align}
Let us set $r=\mathrm{rk}_\delta\left(\mathbf{H}_{L,0}\right)$, $\mathcal{C}=\sqrt{r(\min\{nL,N(T-L)\}-r)}$, and let us write $h_{j,k,1}$ to denote the $k$-row of the $j$-row block $h_{j,1}$ of $\mathbf{H}_{L,1}$ for $j=1,\ldots,L$ and $k=1,\ldots,n$. By \eqref{eq:ID-cond-1} and \eqref{eq:ID-cond-2}, applying corollary \ref{cor:SpLSSolver} to each pair $\mathbf{H}_{L,0}^\top,h_{j,k,1}^\top$ we can compute an orthogonal projector $\hat{Q}$ and a vector $a_{j,k,T}\in \mathbb{C}^{1\times (N(T-L))}$ such that 
\begin{align}
\left\|\mathbf{H}_{L,0}^\top a_{j,k,T}^\top-h_{j,k,1}^\top\right\|&\leq \mathcal{C}\left \|a_{j,k,T}^\top\right\|\delta\nonumber\\
 &+\left\|(I_{N(T-L)}-\hat{Q})h_{j,k,1}^\top\right\|\nonumber\\
\label{eq:err-bound-1}
\end{align}
and $a_{j,k,T}$ has at most $r$ nonzero entries for each $j=1,\ldots,L$ and each $k=1,\ldots,n$. Let us set
\begin{align*}
\hat{A}_T&=\begin{bmatrix}
a_{1,1,T}^\top & \cdots & a_{1,n,T}^\top & \cdots & a_{L,1,T}^\top & \cdots & a_{L,n,T}^\top
\end{bmatrix}^\top,\\
Q&=\hat{Q}^\top
\end{align*}
and for each $t=1,\ldots,T-L$ and each $g_j\in G_N$, let us set
\begin{align*}
X_t&=\begin{bmatrix}
x_t^\top & \cdots & x_{t+L-1}^\top
\end{bmatrix}^\top\\
Y_{j,t}&=I_L\otimes g_jX_t.
\end{align*}
It can be easily verified that $Q$ is an orthogonal projector, and by \eqref{eq:err-bound-1} if we set $\mathbb{D}=\sqrt{nL}\|\hat{A}_T\|_F\mathcal{C}$ we will have that
\begin{align}
\left\|\mathbf{H}_{L,1}-\hat{A}_T\mathbf{H}_{L,0}\right\|_F&\leq \mathbb{D}\delta+\sqrt{nL}\left\|\mathbf{H}_{L,1}(I_{N(T-L)}-\hat{Q}^\top)\right\|_F\nonumber\\
&=\mathbb{D}\delta+\sqrt{nL}\left\|\mathbf{H}_{L,1}(I_{N(T-L)}-Q)\right\|_F.
\label{eq:main-bound-1}
\end{align}

Let us set $\mathbb{E}=\mathbb{D}\left(\sum_{t=0}^{T-1} \|\hat{A}_T\|_F^t\right)$ and $\mathbb{F}=\sqrt{nL}\left(\sum_{t=0}^{T-1} \|\hat{A}_T\|_F^t\right)$, by \eqref{eq:main-bound-1} and remark \ref{rem:Hankel-block} we will have that
\begin{align*}
\|X_{t+1}-\hat{A}_TX_t\|&\leq \left\|\mathbf{H}_{L,1}-\hat{A}_T\mathbf{H}_{L,0}\right\|_F,\\
&\leq  \mathbb{D}\delta+\sqrt{nL}\left\|\mathbf{H}_{L,1}(I_{N(T-L)}-Q)\right\|_F\\
\left\|Y_{j,t+1}-\hat{A}_TY_{j,t}\right\|&\leq\left\|\mathbf{H}_{L,1}-\hat{A}_T\mathbf{H}_{L,0}\right\|_F\\
&\leq \mathbb{D}\delta+\sqrt{nL}\left\|\mathbf{H}_{L,1}(I_{N(T-L)}-Q)\right\|_F
\end{align*}
and this implies that
\begin{align}
\left\|X_{t+1}-\hat{A}_T^tX_1\right\|&\leq \left\|X_{t+1}-\hat{A}_{T}X_{t}\right\|+\left\|\hat{A}_{T}X_{t}-\hat{A}_T^tX_1\right\|\nonumber\\
&\leq \mathbb{D}\delta+\sqrt{nL}\left\|\mathbf{H}_{L,1}(I_{N(T-L)}-Q)\right\|_F\nonumber\\
&+\|\hat{A}_T\|_F\left\|X_{t}-\hat{A}_T^{t-1}X_1\right\|\nonumber\\
&\cdots\leq \mathbb{E}\delta+\mathbb{F}\left\|\mathbf{H}_{L,1}(I_{N(T-L)}-Q)\right\|_F
\label{eq:main-bound-02}
\end{align}
and
\begin{align}
\left\|Y_{j,t+1}-\hat{A}_T^tY_{j,1}\right\|&\leq \left\|Y_{j,t+1}-\hat{A}_{T}Y_{j,t}\right\|+\left\|\hat{A}_{T}Y_{j,t}-\hat{A}_T^tY_{j,1}\right\|\nonumber\\
&\leq \mathbb{D}\delta+\sqrt{nL}\left\|\mathbf{H}_{L,1}(I_{N(T-L)}-Q)\right\|_F\nonumber\\
&+\|\hat{A}_T\|_F\left\|Y_{j,t}-\hat{A}_T^{t-1}Y_{j,1}\right\|\nonumber\\
&\cdots\leq \mathbb{E}\delta+\mathbb{F}\left\|\mathbf{H}_{L,1}(I_{N(T-L)}-Q)\right\|_F.
\label{eq:main-bound-03}
\end{align}

Let us set 
\begin{align*}
A_T=\frac{1}{N}\sum_{j=1}^{N}\left(I_L\otimes g_j^\ast\right)\hat{A}_T\left(I_L\otimes g_j\right).
\end{align*}
Since $G_N=\{g_1,\ldots,g_N\}$ is a finite group of unitary matrices and $N=\sharp(G_N)$, we will have that for each $1\leq j,k\leq N$, there is $1\leq l\leq N$ such that
\begin{align*}
I_L\otimes g_j \: I_L\otimes g_k=I_L\otimes g_jg_k=I_L\otimes g_l
\end{align*}
and by elementary group representation theory this implies that for each $1\leq k\leq N$
\begin{align*}
\left(I_L\otimes g_k\right)^\ast A_T\left(I_L\otimes g_k\right)= I_L\otimes g_k^\ast A_T I_L\otimes g_k=A_T.
\end{align*}
Consequently, 
\begin{align}
I_L\otimes g_k A_T=A_T I_L\otimes g_k
\label{eq:cm-rel-1}
\end{align}
for each $g_k\in G_N$, and this implies that $A_T\in \mathbb{S}(nL)^{\tilde{G}_N}$ with $\tilde{G}_N$ defined as in problem \ref{prob:prob-1}.

By definition of $A_T$ we will have that
\begin{align}
\|A_T\|_F&=\left\|\frac{1}{N}\sum_{j=1}^{N}\left(I_L\otimes g_j^\ast\right)\hat{A}_T\left(I_L\otimes g_j\right)\right\|F\nonumber\\
&\leq \frac{1}{N}\sum_{j=1}^N \|\hat{A}_T\|_F=\|\hat{A}_T\|_F
\label{eq:P-ineq}
\end{align}

By definition of $\mathbf{H}_{L,0},\mathbf{H}_{L,1}$ and by \eqref{eq:main-bound-1} and \eqref{eq:cm-rel-1} we will have that
\begin{align*}
\|X_{t+1}-A_TX_t\|&=\left\|\sum_{j=1}^N\frac{1}{N}X_{t+1}-\frac{1}{N}\sum_{j=1}^N I_L\otimes g_j^\ast \hat{A}_T I_L\otimes g_jX_t \right\|\\
&\leq\frac{1}{N}\sum_{j=1}^N \left\|I_L\otimes g_jX_{t+1}-\hat{A}_TI_L\otimes g_jX_t\right\|\\
&\leq\frac{1}{N}\sum_{j=1}^N \left\|\mathbf{H}_{L,1}-\hat{A}_T\mathbf{H}_{L,0}\right\|_F\\
&= \left\|\mathbf{H}_{L,1}-\hat{A}_T\mathbf{H}_{L,0}\right\|_F\\
&\leq \mathbb{D}\delta+\sqrt{nL}\left\|\mathbf{H}_{L,1}(I_{N(T-L)}-Q)\right\|_F\\
\|Y_{j,t+1}-A_TY_{j,t}\|&=\left\|I_L\otimes g_jX_{t+1}-A_TI_L\otimes g_jX_t\right\|\\
&=\left\|I_L\otimes g_j\left(X_{t+1}-A_TX_t\right)\right\|\\
&=\left\|X_{t+1}-A_TX_t\right\|\\
&\leq \mathbb{D}\delta+\sqrt{nL}\left\|\mathbf{H}_{L,1}(I_{N(T-L)}-Q)\right\|_F.
\end{align*}
Consequently, by \eqref{eq:P-ineq} we will have that
\begin{align}
\left\|X_{t+1}-A_T^tX_1\right\|&\leq \left\|X_{t+1}-A_{T}X_{t}\right\|+\left\|A_{T}X_{t}-A_T^tX_1\right\|\nonumber\\
&\leq \mathbb{D}\delta+\sqrt{nL}\left\|\mathbf{H}_{L,1}(I_{N(T-L)}-Q)\right\|_F\nonumber\\
&+\|\hat{A}_T\|_F\left\|X_{t}-A_T^{t-1}X_1\right\|\nonumber\\
&\cdots\leq \mathbb{E}\delta+\mathbb{F}\left\|\mathbf{H}_{L,1}(I_{N(T-L)}-Q)\right\|_F
\label{eq:main-bound-2}
\end{align}
and
\begin{align}
\left\|Y_{j,t+1}-A_T^tY_{j,1}\right\|&\leq \left\|Y_{j,t+1}-A_{T}Y_{j,t}\right\|+\left\|A_{T}Y_{j,t}-A_T^tY_{j,1}\right\|\nonumber\\
&\leq \mathbb{D}\delta+\sqrt{nL}\left\|\mathbf{H}_{L,1}(I_{N(T-L)}-Q)\right\|_F\nonumber\\
&+\|\hat{A}_T\|_F\left\|Y_{j,t}-A_T^{t-1}Y_{j,1}\right\|\nonumber\\
&\cdots\leq \mathbb{E}\delta+\mathbb{F}\left\|\mathbf{H}_{L,1}(I_{N(T-L)}-Q)\right\|_F.
\label{eq:main-bound-3}
\end{align}
Since $\mathscr{T}(x_t)=x_{t+1}$ and since it is clear that $\mathscr{P}_L$ is a partial isometry that satisfies $x_t=\mathscr{P}_LX_t$ and $g_jx_t=\mathscr{P}_L I_L\otimes g_j X_t$ for each $t=1,\ldots,T-L$ and each $g_j\in G_N$, by \eqref{eq:main-bound-2} we will have that
\begin{align*}
\left\|\mathscr{T}(x_t)-\mathscr{P}_L\hat{A}_T^tX_1\right\|&=\left\|\mathscr{P}_L(X_{t+1}-\hat{A}_T^tX_1)\right\|\\
&\leq \left\|X_{t+1}-\hat{A}_T^tX_1\right\|\\
&\leq \mathbb{E}\delta+\mathbb{F}\left\|\mathbf{H}_{L,1}(I_{N(T-L)}-Q)\right\|_F,\\
\left\|\mathscr{T}(x_t)-\mathscr{P}_LA_T^tX_1\right\|&=\left\|\mathscr{P}_L(X_{t+1}-A_T^tX_1)\right\|\\
&\leq \left\|X_{t+1}-A_T^tX_1\right\|\\
&\leq \mathbb{E}\delta+\mathbb{F}\left\|\mathbf{H}_{L,1}(I_{N(T-L)}-Q)\right\|_F
\end{align*}
and by  \eqref{eq:main-bound-3} we will also have that
\begin{align*}
\left\|\mathscr{T}(g_jx_t)-\mathscr{P}_L\hat{A}_T^t (I_L\otimes g_j) X_1\right\|&=\left\|g_j\mathscr{T}(x_t)-\mathscr{P}_L\hat{A}_T^t Y_{j,1}\right\|\\
&=\left\|\mathscr{P}_L(Y_{j,t+1}-\hat{A}_T^tY_{j,1})\right\|\\
&\leq \left\|Y_{j,t+1}-\hat{A}_T^tY_{j,1}\right\|\\
&\leq \mathbb{E}\delta\\
&+\mathbb{F}\left\|\mathbf{H}_{L,1}(I_{N(T-L)}-Q)\right\|_F,\\
\left\|\mathscr{T}(g_jx_t)-\mathscr{P}_LA_T^t (I_L\otimes g_j) X_1\right\|&=\left\|g_j\mathscr{T}(x_t)-\mathscr{P}_LA_T^t Y_{j,1}\right\|\\
&=\left\|\mathscr{P}_L(Y_{j,t+1}-A_T^tY_{j,1})\right\|\\
&\leq \left\|Y_{j,t+1}-A_T^tY_{j,1}\right\|\\
&\leq \mathbb{E}\delta\\
&+\mathbb{F}\left\|\mathbf{H}_{L,1}(I_{N(T-L)}-Q)\right\|_F.
\end{align*}
This completes the proof.
\end{proof}

We can apply theorem \ref{thm:main-thm} for sparse model identification, specially when a sparse predictive model $\hat{\mathscr{T}}$ for a time series $\{x_t\}_{t\geq 1}$ of a system $(\Sigma,\mathscr{T})$ can be estimated for some given time horizon $T>0$, based on a relatively small sample $\Sigma_{S}=\{x_t\}_{t=1}^{S}\subset \{x_t\}_{t\geq 1}$ for some $S \ll T$. Important examples of systems that would satisfy the previous consideration are the periodic and eventually periodic systems, like the ones considered in \cite{FRANKE200364} and \cite{FARHOOD2002417}, respectively.

\subsubsection{An intuitive topological approach to sparse identification of time series models}
\label{section:topological_control}

Let us consider the following set.

\begin{defn}
Given $\delta>0$, a finite group $G_N\subset \mathbb{U}(n)$ with $\sharp(G_N)=N$, and a sample $\Sigma_T=\{x_t\}_{t=1}^{T}$ from a time series $\{x_t\}_{t\geq 1}\subset \mathbb{C}^n$ of a $G_N$-equivariant system $(\Sigma,\mathscr{T})$. We will write $\mathrm{Gr}_{\delta,\Sigma_T}^{G_N}$ to denote the {\em $\delta$-approximate symmetric identification grading set} of the sample $\Sigma_T$ that will be defined by the following expression
\begin{equation}
\mathrm{Gr}_{\delta,\Sigma_T}^{G_N}=\left\{1\leq L\leq \left\lfloor\frac{T+1}{2}\right\rfloor\left|\begin{array}{l}
\mathrm{rk}_\delta\left(\mathbf{H}^{L+1}_{T}\right)=\mathrm{rk}_\delta\left(\mathbf{H}^{L}_{T-1}\right),\\
\mathrm{rk}_\delta\left(\mathbf{H}^{L}_{T-1}\right)>0
\end{array}
\right.\right\}
\label{eq:grad-def}
\end{equation}
with $\mathbf{H}^{L+1}_{T}=\mathscr{H}_{L+1}\left(\Sigma_{T},G_N\right)$ and  $\mathbf{H}^{L}_{T-1}=\mathscr{H}_{L}\left(\{x_t\}_{t=1}^{T-1},G_N\right)$.
\end{defn}

Let us now consider a nonnegative integer defined as follows.

\begin{defn}
Given $\delta>0$, a finite group $G_N\subset \mathbb{U}(n)$ with $\sharp(G_N)=N$, and a sample $\Sigma_T=\{x_t\}_{t=1}^{T}$ from a time series $\{x_t\}_{t\geq 1}\subset \mathbb{C}^n$ of a $G_N$-equivariant system $(\Sigma,\mathscr{T})$. We will write $\mathrm{deg}_{\delta,G_N}(\Sigma_T)$ to denote the {\em $\delta$-approximate symmetric identification degree} of the sample $\Sigma_T$ that will be defined by the following expression.
\begin{equation}
\mathrm{deg}_{\delta,G_N}(\Sigma_T)=\left\{\begin{array}{ll}
\inf \: \mathrm{Gr}_{\delta,\Sigma_T}^{G_N} &, \: \sharp(\mathrm{Gr}_{\delta,\Sigma_T}^{G_N})\neq 0\\
0 &,\: \sharp(\mathrm{Gr}_{\delta,\Sigma_T}^{G_N})= 0
\end{array}
\right.
\label{eq:degree-def}
\end{equation}
\end{defn}

From a topological perspective the approximate symmetric identification degree provides a way to compute invariants that can be used to classify samples from time series data in terms of the models that are computable based on such samples. More specifically, given a finite group $G_N\subset \mathbb{U}(n)$ with $\sharp(G_N)=N$, and a sample $\Sigma_{T}=\{x_t\}_{t=1}^T$ from a time series $\{x_t\}_{t\geq 1}\subset \mathbb{C}^n$ of a $G_N$-equivariant system $(\Sigma,\mathscr{T})$ such that $\mathrm{deg}_{\delta,G_N}(\Sigma_T)=L>0$ for some $\delta>0$, we will have that if we set $\Sigma_{T-1}=\{x_{t}\}_{t=1}^{T-1}$ then
\begin{align}
\mathrm{rk}_\delta\left(\mathscr{H}_{L+1}\left(\Sigma_T,G_N\right)\right)=\mathrm{rk}_\delta\left(\mathscr{H}_{L}\left(\Sigma_{T-1},G_N\right)\right)>0.
\label{eq:rank-cond}
\end{align}
By \eqref{eq:rank-cond} and as a consequence of theorems \ref{thm:MEQ-Solver}, \ref{thm:main-thm} and the ideas implemented in their proofs, we will have that there are orthogonal projectors $P,Q$ such that 
\begin{align}
\mathrm{rk}(P)=\mathrm{rk}(Q)=\mathrm{rk}_\delta\left(\mathscr{H}_{L}\left(\Sigma_{T-1},G_N\right)\right)>0
\label{eq:rank-homot-cond}
\end{align}
and
\begin{align}
&\|P\mathscr{H}_{L}\left(\Sigma_{T-1},G_N\right)^\top-\mathscr{H}_{L}\left(\Sigma_{T-1},G_N\right)^\top\|_F\leq C_1 \delta,\nonumber\\
&\|Q\mathscr{H}_{L+1}\left(\Sigma_T,G_N\right)^\top-\mathscr{H}_{L+1}\left(\Sigma_T,G_N\right)^\top\|_F\leq C_2 \delta,
\label{eq:almost-units-cond}
\end{align}
for some constants $C_1,C_2\geq 0$. By \eqref{eq:rank-homot-cond} we will have that as a consequence of the Schur decomposition theorem, there are unitaries $U_P,U_Q$ such that
\begin{align*}
U_P^\ast PU_P=U_QQU_Q^\ast
\end{align*}
this in turn implies that
\begin{align*}
P=U_PU_QQ(U_PU_Q)^\ast
\end{align*}
and by elementary Lie group theory there is an analytic path of unitaries $U(s)=e^{s\log(U_PU_Q)}$ from the identity matrix $I_{N(T-L)}=U(0)$ to $U_PU_Q=e^{\log(U_PU_Q)}=U(1)$, and this implies that there is an analytic path of orthogonal projectors $Q(s)=U(s)QU(s)^\ast$ from $Q$ to $P$. Consequently, the orthogonal projectors $P$ and $Q$ are homotopic in the sense that there is a homotopy between the corresponding constant maps over $P$ and $Q$, respectively. In addition, by \eqref{eq:embedding-ID-cond-2} and \eqref{eq:almost-units-cond} we will have that $P$ and $Q$ satisfy the following norm constraint.
\begin{align*}
&\|(P-Q)\mathscr{H}_{L}\left(\Sigma_{T-1},G_N\right)^\top\|_F\leq (C_1+C_2) \delta
\end{align*}

Based on the previous considerations, we can observe that the number $\mathrm{drk}_{\delta,G_N}(\Sigma_T)$ defined by the expression
\begin{align*}
\mathrm{drk}_{\delta,G_N}(\Sigma_T)=\mathrm{rk}_\delta\left(\mathscr{H}_{L+1}\left(\Sigma_T,G_N\right)\right)-\mathrm{rk}_\delta\left(\mathscr{H}_{L}\left(\Sigma_{T-1},G_N\right)\right)
\end{align*}
provides a way of measuring potential obstructions for the computability of predictive models based on $\Sigma_T$ that can reach a prediction error $\mathcal{O}(\delta)$, in the sense that it would be necessary for $\mathrm{drk}_{\delta,G_N}(\Sigma_T)$ to be equal to zero in order for the topological obstruction to be removed. A more formal study of these potential topologically controlled obstructions together with their connections with potential model overfitting, will be further studied in future communications.

\begin{ex} As an example of the previous phenomena, let us consider a sample $\Sigma_{257}=\{s_k:1\leq k\leq 257\}$ from a discrete time scalar signal $S=\{s_k:k\in \mathbb{Z}^+\}$ with trivial symmetry group $G_1=\{1\}$, that is determined for each $k\in \mathbb{Z}^+$ by the expression 
\begin{align*}
s_k=\sum_{j=0}^\infty \min\{t_k-j,1-t_k+j\}(H_{j}(t_k)-H_{j+1}(t_k))
\end{align*}
with $t_k=(k-1)/32$. Let us add noise to the sample $\Sigma_{257}$ using a sequence of normally distributed pseudorandom numbers $\mathbf{RN}_{257}=\{r_{k}:|r_k|=\mathcal{O}(1\times 10^{-3}), 1\leq k\leq 257\}$, obtaining a noisy version $\tilde{\Sigma}_{257}=\{s_k+r_k:1\leq k\leq 257\}$ of the original sample $\Sigma_{257}$. 

Let us consider a subsample $\tilde{\Sigma}_{70}=\{s_k+r_k:1\leq k\leq 70\}\subset \tilde{\Sigma}_{257}$. Computing $\mathrm{deg}_{\delta,G_1}(\tilde{\Sigma}_{70})$ with $\delta=1\times 10^{-2}$ with the Matlab program {\tt deg.m} in \cite{CodeVides} we obtain $\mathrm{deg}_{\delta,G_1}(\tilde{\Sigma}_{70})=17$. We can now compute a predictive model 
\begin{align*}
s_{k+1}=c_{1}s_{k}+c_{2}s_{k-1}+\cdots+c_{L}s_{k-L+1}\approx\mathscr{P}_LA_T^k\begin{bmatrix}
s_{1} \\ s_{2} \\ \vdots \\ s_{L}\end{bmatrix}
\end{align*}
for $k\geq L=\mathrm{deg}_{\delta,G_1}(\tilde{\Sigma}_{70})=17$, along the lines of the proof of theorem \ref{thm:main-thm} obtaining the following model
\begin{align*}
s_{k+1}&=1.000036116667133s_{k}-0.998153385290551s_{k-15}\\
&+0.998115019553671s_{k-16}
\end{align*}
for $k\geq 17$. 

The identified signals for different values of the lag parameter $L$ are shown in figure \ref{fig:Example1Signals}.
\begin{figure}[h]
\centering
\includegraphics[scale=.55]{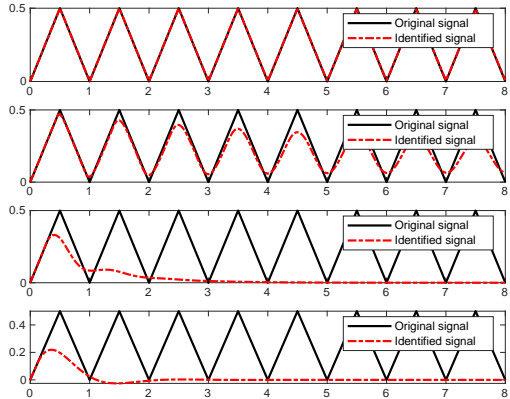}
\caption{Identified signals for different lag values. $L=17$ (top first row). $L=16$ (second row). $L=10$ (third row). $L=5$ (bottom row).}
\label{fig:Example1Signals}
\end{figure}

The $c_k$ coefficients corresponding to the models identified for different values of the lag parameter $L$ are shown in figure \ref{fig:Example1ModelCoefficients}.
\begin{figure}[h]
\centering
\includegraphics[scale=.55]{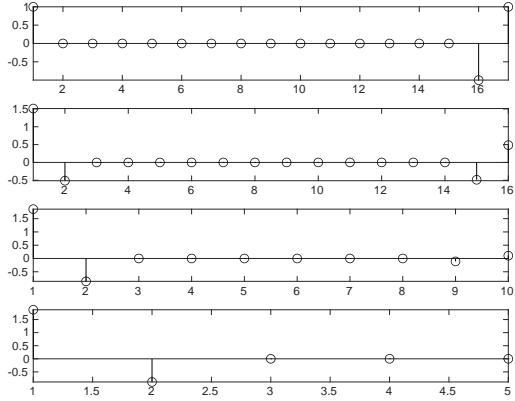}
\caption{Identified coefficients for different lag values. $L=17$ (top first row). $L=16$ (second row). $L=10$ (third row). $L=5$ (bottom).}
\label{fig:Example1ModelCoefficients}
\end{figure}
The root mean square error $RMSE$ corresponding to different values of the lag parameter $L$ are documented in table \ref{table:RMSEtable}.
\begin{center}
 \begin{tabular}{||c | c||} 
 \hline
 Lag value $L$ & RMSE\\ [0.5ex] 
 \hline\hline
 17 &  0.002848195208845\\ 
 \hline\hline 
 16 & 0.067357195429571\\
 \hline
 \hline\hline 
 10 & 0.265912048561782\\
 \hline\hline 
 5 & 0.279715847089058\\
 \hline
 \end{tabular}
\captionof{table}{Root mean square errors corresponding to each lag value.}
\label{table:RMSEtable}
\end{center}
The computational setting used for this example is documented in the program {\tt Example1.m} in \cite{CodeVides}. For each lag value the corresponding model parameters were computed using program {\tt SpSolver.m} in \cite{CodeVides} based on algorithm \ref{alg:main_SLMESolver_alg_1}, with the same tolerance $\delta>0$, in order to expose the potential topologically controlled approximation obstruction identified by the number $\mathrm{drk}_{\delta,G_1}(\tilde{\Sigma}_{70})$.
\end{ex}

As a consequence of theorem \ref{thm:main-thm} and the previous observations, it can be seen that given a finite group $G_N\subset \mathbb{U}(n)$ and a sample $\Sigma_{T}=\{x_t\}_ {t=1}^T$ from a time series $\{x_t\}_{t\geq 1}\subset \mathbb{C}^n$ of a $G_N$-equivariant system $(\Sigma,\mathscr{T})$, if the number $\hat{L}=\mathrm{deg}_{\delta,G_N}(\Sigma_T)$ is positive, we can use this number to estimate a necessary condition for the computability of a sparse solution to the problem
\begin{align*}
\mathbf{H}_{L,0}^\top A_T^\top \approx_\delta \mathbf{H}_{L,1}^\top
\end{align*}
where $\mathbf{H}_{L,k}=\mathscr{H}_{L}\left(\Sigma_{k},G_N\right)$ for $k=0,1$, with $\Sigma_{0}=\{x_t\}_{t=1}^{T-1}$ and $\Sigma_{1}=\{x_t\}_{t=2}^{T}$. In particular, when $\hat{L}=\mathrm{deg}_{\delta,G_N}(\Sigma_T)$ is positive, the lag value $L=\hat{L}$ would provide a good starting point for an adaptive sparse system identification method, if the prediction error is still not small enough for the lag value $L=\hat{L}$, a controlling algorithm can keep increasing the value of the lag parameter until the prediction error is reached or some prescribed bound for the lag value is attained. For the experiments documented in this article, we have used standard scalar signal autocorrelation techniques to estimate admissible bounds for the lag values, in the Sparse Dynamical System Identification ({\bf SDSI}) toolset available in \cite{CodeVides} these ideas are implemented in the Matlab program {\tt LagEstimate.m} based on the Matlab function {\tt xcorr.m}.

The results and ideas presented in this section can be translated into sparse identification algorithms like algorithm \ref{alg:main_alg_1}.

\subsection{Sparse parameter identification for finite difference models}

Let us write $\eth_{h,k}$ to denote the finite difference representation of the time differentiation operator $\partial_t$ with approximation order $k$ and uniform time partition size $h$. Given a dictionary of dynamic variables $u_1,\ldots,u_m$, a function $f:(\mathbb{C}^n)^m\to (\mathbb{C}^n)^p$ and time series samples $\{u_j(k)\}_{k=1}^N\subset \mathbb{C}^n$ corresponding to each dynamic variable, let us write $V_N(f)(\mathbf{u})$ to denote the expression.
\begin{align*}
V_N(f)(\mathbf{u})&=\begin{bmatrix}
f(u_1(1),\ldots,u_m(1))\\
\vdots\\
f(u_1(N),\ldots,u_m(N))
\end{bmatrix}
\end{align*}
In this section, given $\delta>0$ we will consider the approximate sparse identification problems of the form
\begin{align}
\begin{bmatrix}
| &  & |\\
\hat{\eth}_{h,k}^1(\hat{u}_1) & \cdots& \hat{\eth}_{h,k}^m(\hat{u}_m)\\
| &  & |
\end{bmatrix}\approx_\delta\begin{bmatrix}
| &  & |\\
V_N(f_1)(\mathbf{u}) & \cdots & V_N(f_M)(\mathbf{u})\\
| &  & |
\end{bmatrix}\mathbf{C},
\label{eq:dif-model-id-problem}
\end{align}
where $\hat{\eth}_{k,h}^1,\ldots,\hat{\eth}_{k,h}^m$ denote finite difference operators based on $\eth_{k,h}$ with some additional specifications determined by the geometric or physical configuration of the system under study, and for each $1\leq k\leq m$, $\hat{u}_k=V_N(p_k)(\mathbf{u})$ for $p_k$ determined by the expression $p_k(x_1,\ldots,x_m)=x_k$, and where each function $f_j:(\mathbb{C}^n)^{m_j}\to (\mathbb{C}^n)^{p_j}$ is given for $j=1,\ldots,M$.

By applying theorem \ref{thm:MEQ-Solver} to \eqref{eq:dif-model-id-problem} we obtain the following solvability result.

\begin{cor}\label{cor:s-dif-model-id}
Given $\delta>0$ and a problem of the form \eqref{eq:dif-model-id-problem}. If we set 
\begin{align*}
r=\mathrm{rk}_\delta\left(\begin{bmatrix}
| &  & |\\
V_N(f_1)(\mathbf{u}) & \cdots & V_N(f_M)(\mathbf{u})\\
| &  & |
\end{bmatrix}\right)
\end{align*}
and if $r>0$ then, there is a solution $\mathbf{C}$ to the problem \eqref{eq:dif-model-id-problem} with at most $r\sum_{k=1}^Mp_k$ nonzero entries.
\end{cor}
\begin{proof}
This is a direct application of theorem \ref{thm:MEQ-Solver} to problem \eqref{eq:dif-model-id-problem}.
\end{proof}

Once an approximate sparse solution $\mathbf{C}$ to the problem \eqref{eq:dif-model-id-problem} has been computed, one can apply numerical time integration methods to the continuous-time approximate representation of \eqref{eq:dif-model-id-problem} determined by the expression
\begin{align*}
\begin{bmatrix}
| &  & |\\
\partial_t(\hat{u}_1) & \cdots& \partial_t(\hat{u}_m)\\
| &  & |
\end{bmatrix}=\begin{bmatrix}
| &  & |\\
f_1(\mathbf{u}) & \cdots & f_M(\mathbf{u})\\
| &  & |
\end{bmatrix}\mathbf{C},
\end{align*}
with $f_j(\mathbf{u})=f_j(u_1,\dots,u_m)$, in order to obtain a collection $\mathscr{T}_1,\ldots,\mathscr{T}_m$ of transition operators that approximately satisfy the equations:
\begin{align*}
u_j(k+1)=\mathscr{T}_j(u_1(k),\ldots,u_m(k))
\end{align*}
for each $j=1,\ldots,m$, and each discrete time index $k\geq 1$.

\section{Computational Methods}

\subsection{Algorithms}
\label{section_algorithms}

One of the purposes of this project is to provide sparse dynamical system identification ({\bf SDSI}) tools that can be used to build collaborative frameworks of theoretical and computational methods that can be applied in a multidisciplinary context where adaptive approximate system identification is required. An example of the aforementioned collaborative frameworks can be described by the automaton illustrated in figure \ref{alg:main_alg_0}. 
\begin{figure}[!htp]
\begin{center}
\begin{tikzpicture}[->,>=stealth',shorten >=1pt,auto,node distance=2.8cm,
                    semithick]
\node[state, initial] (1) {$D$};
\node[state, below left of=1] (2) {$M$};
\node[state, right of=2] (3) {$P$};
\draw (1) edge[above, bend right] node{$1$} (2)
(1) edge[loop above] node{$0,2$} (1)
(2) edge[loop left] node{$0$} (2)
(2) edge[above, right] node{$2$} (1)
(2) edge[below,bend right] node{$1$} (3)
(3) edge[left,below,above] node{$1$} (2)
(3) edge[above,bend right] node{$2$} (1)
(3) edge[loop right,right] node{$0$} (3);
\end{tikzpicture}
\caption{Finite-state automaton description of generic {\bf SDSI} processes.}
\label{alg:main_alg_0}
\end{center}
\end{figure}
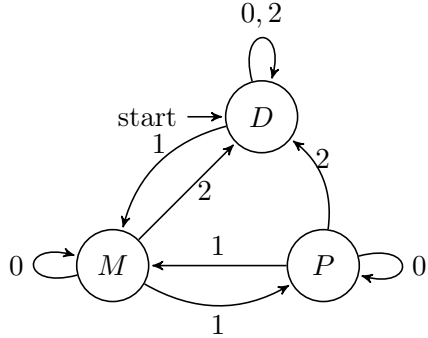

The blocks $D$, $M$ and $P$ of the system \ref{alg:main_alg_0} correspond to the data processing, model computation and predictive simulation stages of a generic sparse system identification process, respectively, while the labels $0$, $1$ and $2$ correspond to the states {\em computation in progress}, {\em computation completed} and {\em more data are required}, respectively.

In this document we focus on the sparse linear least squares solver algorithms and approximate topological invariants in the form of easily computable numbers, that can be used in the modeling block $M$ of the automaton \ref{alg:main_alg_0} for sparse model identification. Among other cases, the control signals for automata like \ref{alg:main_alg_0} can be provided by an expert interested on the dynamics identification of some particular system, or by an artificially intelligent control system designed to build digital twins for a given system or process in some industrial environment. Although the programs in \cite{CodeVides} can be used, adapted or modified to work in any of the two cases previously considered, the programs and examples included as part of the work reported in this document are written with the first case in mind. The artificially intelligent schemes will be further explored in the future.

Although the results in this document focus on sparse model identification, besides the programs corresponding to sparse linear least squares solvers and approximate degree and rank identifiers based on the results in \S\ref{section:linear-solvers} and \S\ref{section_SDSI}, respectively, some programs for data reading and writing, synthetic signals generation, and predictive simulation are also include as part of the {\bf SDSI} toolset available in \cite{CodeVides}.

\subsubsection{Sparse linear least squares solver algorithm}

As an application of the results and ideas presented in \S\ref{section:linear-solvers} one can obtain a prototypical sparse linear least squares solver algorithm like algorithm \ref{alg:main_SLMESolver_alg_1}.

\begin{figure}[!tph]
\begin{algorithm}[H]
\begin{flushleft}
\caption{{\bf SLRSolver}: Sparse linear least squares solver algorithm}
\label{alg:main_SLMESolver_alg_1}
\end{flushleft}
\begin{algorithmic}
\STATE{{\bf Data:}\:\:\: $A\in \mathbb{C}^{m\times n}$, $Y\in \mathbb{C}^{m\times p}$, $\delta>0$, $N\in \mathbb{Z}^+$, $\varepsilon>0$}
\STATE{{\bf Result:}\:\:\: $X=\mathbf{SLRSolver}(A,Y,\delta,N,\varepsilon)$}
\begin{enumerate}
\STATE{Compute economy-sized SVD $USV=A$\;}
\STATE{Set $s=\min\{m,n\}$\;}
\STATE{Set $r=\mathrm{rk}_\delta(A)$\;}
\STATE{Set $U_\delta=\sum_{j=1}^r U\hat{e}_{j,s}\hat{e}_{j,s}^\ast$\;}
\STATE{Set $T_\delta=\sum_{j=1}^r (\hat{e}_{j,s}^\ast S\hat{e}_{j,s})^{-1} \hat{e}_{j,s} \hat{e}_{j,s}^\ast$\;}
\STATE{Set $V_\delta=\sum_{j=1}^r \hat{e}_{j,s}\hat{e}_{j,s}^\ast V$\;}
\STATE{Set $\hat{A}=U_\delta^\ast A$\;}
\STATE{Set $\hat{Y}=U_\delta^\ast Y$\;}
\STATE{Set $X_0=V_\delta^\ast T_\delta \hat{Y}$\;}
\FOR{$j=1,\ldots,p$}
\STATE{Set $K=1$\;} 
\STATE{Set $\mathrm{error}=1+\delta$\;}
\STATE{Set $c=X_0\hat{e}_{j,p}$}
\STATE{Set $x_0=c$}
\STATE{Set $\hat{c}=\begin{bmatrix}
\hat{c}_1 & \cdots & \hat{c}_n
\end{bmatrix}^\top=\begin{bmatrix}
|\hat{e}_{1,n}^\ast c| & \cdots & |\hat{e}_{n,n}^\ast c|
\end{bmatrix}^\top$\;}
\STATE{Compute permutation $\sigma:\{1,\ldots,n\}\to \{1,\ldots,n\}$ such that: $\hat{c}_{\sigma(1)}\geq \hat{c}_{\sigma(2)}\geq \cdots \geq \hat{c}_{\sigma(n)}$}
\STATE{Set $N_0=\max\left\{\sum_{j=1}^n H_\varepsilon\left(\hat{c}_{\sigma(j)}\right),1\right\}$\;}
\WHILE{$K\leq N$ \AND $\mathrm{error}>\delta$} 
\STATE{Set $x=\mathbf{0}_{n,1}$\;}
\STATE{Set $A_0=\sum_{j=1}^{N_0} \hat{A}\hat{e}_{\sigma(j),n}\hat{e}_{j,N_0}^\ast$}
\STATE{Solve $c=\arg\min_{\tilde{c}\in \mathbb{C}^{N_0}}\|A_0\tilde{c}-\hat{Y}\hat{e}_{j,p}\|$\;}
\STATE{
\FOR{$k=1,\ldots,N_0$} 
\STATE{Set $x_{\sigma(k)}=\hat{e}_{k,N_0}^\ast c$\;}
\ENDFOR \;}
\STATE{Set $\mathrm{error}=\|x-x_0\|_\infty$\;}
\STATE{Set $x_0=x$\;}
\STATE{Set $\hat{c}=\begin{bmatrix}
\hat{c}_1 & \cdots & \hat{c}_n
\end{bmatrix}^\top=\begin{bmatrix}
|\hat{e}_{1,n}^\ast x| & \cdots & |\hat{e}_{n,n}^\ast x|
\end{bmatrix}^\top$\;}
\STATE{Compute permutation $\sigma:\{1,\ldots,n\}\to \{1,\ldots,n\}$ such that: $\hat{c}_{\sigma(1)}\geq \hat{c}_{\sigma(2)}\geq \cdots \geq \hat{c}_{\sigma(n)}$}
\STATE{Set $N_0=\max\left\{\sum_{j=1}^n H_\varepsilon\left(\hat{c}_{\sigma(j)}\right),1\right\}$}
\STATE{Set $K=K+1$\;} 
\ENDWHILE
\STATE{Set $x_j=x$\;}
\ENDFOR
\STATE{Set $X=\begin{bmatrix}
| & | &  & |\\
x_1 & x_2 & \cdots & x_p\\
| & | &  & |
\end{bmatrix}$\;}
\end{enumerate}
\RETURN $X$
\end{algorithmic}
\end{algorithm}
\end{figure}

The least squares problems $c=\arg\min_{\hat{c}\in \mathbb{C}^K}\|\hat{A}\hat{c}-y\|$ to be solved as part of the process corresponding to algorithm \ref{alg:main_SLMESolver_alg_1} can be solved with any efficient least squares solver available in the language or program where the sparse linear least squares solver algorithm is implemented. For the Matlab and Julia implementations of algorithm \ref{alg:main_SLMESolver_alg_1} written as part of this research project the {\tt backslash "$\backslash$"} operator is used, and for the Python version of algorithm \ref{alg:main_SLMESolver_alg_1} the function {\tt lstsq} is implemented.

\subsubsection{Sparse time series model identification algorithm}

Given a time series $\{x_t\}_{t\geq 1}$ of a system $(\Sigma,\mathscr{T})$ with transition operator $\mathscr{T}$ to be identified, we can approach the computation of local approximations of $\mathscr{T}$ based on a structured data sample $\Sigma_{T}=\{x_t\}_{t=1}^T\subset \Sigma$ using the prototypical algorithm outlined in algorithm \ref{alg:main_alg_1}.
\begin{figure}[!tph]
\begin{algorithm}[H]
\begin{flushleft}
\caption{{\bf SDSI}: algorithm for sparse dynamical system identification}
\label{alg:main_alg_1}
\end{flushleft}
\begin{algorithmic}
\STATE{{\bf Data:}\:\:\: $G_N\subset\mathbb{U}(n)$, $\hat{\Sigma}_T=\{x_t\}_{t=1}^{T}\subset \mathbb{C}^n$, $L\in \mathbb{Z}^+$, $\delta>0,\varepsilon>0$}
\STATE{{\bf Result:}\:\:\: $(\mathscr{P}_L,\hat{A}_T,A_T,X_1)=\mathbf{SDSI}(G_N,\hat{\Sigma}_T,L,\delta,\varepsilon)$}
\begin{enumerate}
\STATE{Compute $L=\max\{\mathrm{deg}_{\delta,G_N}(\hat{\Sigma}_T),L\}$\;}
\STATE{Set $\hat{\Sigma}_{0}=\{x_t\}_{t=1}^{T-1}$, $\hat{\Sigma}_{1}=\{x_t\}_{t=2}^{T}$\;}
\STATE{Set $\mathbf{H}_{L,k}=\mathscr{H}_{L}(\hat{\Sigma}_k,G_N)$, $k=0,1$\;}
\STATE{Solve $\mathbf{H}_{L,0}^\top C\approx_\delta\mathbf{H}_{L,1}^\top$ applying algorithm \ref{alg:main_SLMESolver_alg_1} with the setting $C=\mathbf{SLRSolver}(\mathbf{H}_{L,0}^\top,\mathbf{H}_{L,1}^\top,\delta,nL,\varepsilon)$\;}
\STATE{Set $\hat{A}_T=C^\top$ \;}
\STATE{Set $A_T=\frac{1}{N}\sum_{j=1}^{N}\left(I_L\otimes g_j^\ast\right)\hat{A}_T\left(I_L\otimes g_j\right)$ \;}
\STATE{Set $X_1=
\left[x_{1}^\top \: \cdots \: x_{L}^\top\right]^\top
$\;}
\STATE{Set $\mathscr{P}_L=\hat{e}_{1,L}^\top\otimes I_n$\;}
\end{enumerate}
\RETURN $(\mathscr{P}_L,\hat{A}_T,A_T,X_1)$
\end{algorithmic}
\end{algorithm}
\end{figure}

For the study reported in this document, when the time series data $\Sigma_{T}=\{x_t\}_{t=1}^T\subset \mathbb{C}^n$ of a given system $(\Sigma,\mathscr{T})$ with transition operator $\mathscr{T}$ to be identified, is not uniformly sampled in time, the sample $\Sigma_{T}$ is preprocessed applying local spline interpolation methods to obtain a uniform in time estimate $\tilde{\Sigma}_{T}=\{\tilde{x}_t\}_{t=1}^T\subset \mathbb{C}^n$ for $\Sigma_{T}$. The Matlab program {\tt DataSpliner.m} is an example of a computational implementation of this interpolation procedure and is included as part of the programs in the {\bf SDSI} toolset available at \cite{CodeVides}.

Given a finite group $G_N\subset \mathbb{U}(n)$, a $G_N$-equivariant system $(\Sigma,\mathscr{T})$ and a structured data sample $\Sigma_T\subset \Sigma$, if the elements $\mathscr{P},\hat{A},A,X_1$ are computed using algorithm \ref{alg:main_alg_1} with the setting $(\mathscr{P},\hat{A},A,X_1)=\mathbf{SDSI}(G_N,\hat{\Sigma}_T,L,\delta,\varepsilon)$ for some suitable $\delta,\varepsilon>0$, one can build two predictive models $\mathbf{P}_{\hat{A}},\mathbf{P}_{A}$ for the time evolution of the system, determined by the recurrence relation $x_{t+1}=\mathscr{T}(x_t), t\geq 1$, using schemes of the form $\mathbf{P}_{\hat{A}}(t)=\mathscr{P}\hat{A}^tX_1$ and $\mathbf{P}_{A}(t)=\mathscr{P}A^tX_1$ for $t\geq 1$.

\subsection{Numerical Simulations}
\label{section_numerics}

In this section we will present some numerical simulations computed using the {\bf SDSI} toolset available in \cite{CodeVides}, that was developed as part of this project, the toolset consists of a collection of programs written in Matlab, Julia and Python that can be used for sparse identification and numerical simulation of dynamical systems. 

The numerical experiments documented in this section were performed with Matlab R2021a (9.10.0.
1602886) 64-bit (glnxa64), Julia 1.6.0, Python 3.8 and Netgen/NGSolve 6.2. Some Matlab and Python Netgen/NGSolve programs were used to generate synthetic data used for some of the system identification processes. All the programs written for synthetic data generation and sparse model identification as part of this project are available at \cite{CodeVides}.

The numerical simulations reported in this section were computed on a Linux Ubuntu Server 20.04 PC equiped with an Intel Xeon E3-1225 v5 (8M Cache, 3.30 GHz) processor and with 40GB RAM.

\subsubsection{Sparse identification of a finite difference model for the nonlinear Schr\"odinger equation in the finite line} \label{exa:example-NLSE}
In this section a finite difference model corresponding to a nonlinear Schr\"odinger equation of the form
\begin{equation}
i\partial_t w +\partial_x^2w+q|w|^2w=0
\label{Schrodinger_eq_def}
\end{equation}
for $x\in [-20,20]$ and $t\geq 0$, will be approximately identified. The synthetic signals corresponding to the csv data file {\tt NLSEqData.csv} that will be used for system identification have been computed using a fourth order Runge-Kutta scheme to integrate the corresponding second order in space finite difference discretization of \eqref{Schrodinger_eq_def}, with initial conditions and homogeneous Dirichlet boundary conditions based on the configuration used by Ramos and Villatoro in \cite{RAMOS199431} using the Matlab program {\tt NLSchrodinger1DRK.m} in \cite{CodeVides}. The amplitudes corresponding to the dynamical behavior data $\Sigma_{4401}=\{\mathbf{w}(t)\}_{t=1}^{4401}\subset \mathbb{C}^{161}$ saved in file {\tt NLSEqData.csv} are visualized in figure \ref{fig:NLSESoliton}.
\begin{figure}[h]
\centering
\includegraphics[scale=.55]{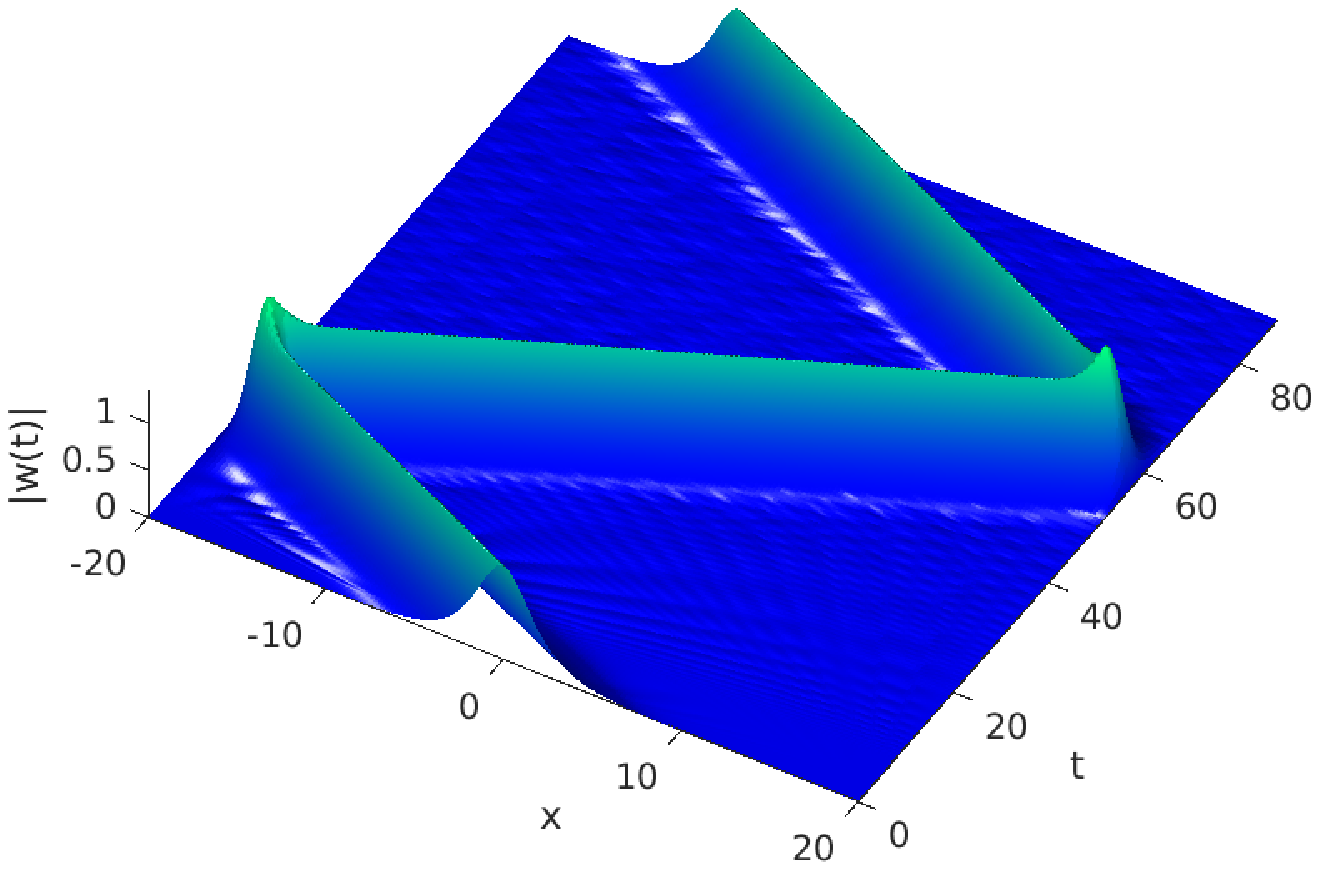}
\caption{Amplitudes corresponding to an approximation of a travelling wave solution to equation \eqref{Schrodinger_eq_def}.}
\label{fig:NLSESoliton}
\end{figure}

Let us consider the finite difference model
\begin{equation}
i\hat{\eth}_{h_t,4}\mathbf{w} +\mathbf{D}_{h_x,2}\mathbf{w}+q|\mathbf{w}|^2\mathbf{w}=0
\label{Schrodinger_eq_def_fd}
\end{equation}
corresponding to \eqref{Schrodinger_eq_def}, where the operation $\hat{\eth}_{h_t,4}\mathbf{w}(t)$ is defined for an arbitrary differentiable function $t\mapsto \mathbf{w}(t)=\begin{bmatrix}w_1(t) & \cdots & w_{161}(t)\end{bmatrix}^\top\in \mathbb{C}^{161}$ by the expression
\begin{align*}
\hat{\eth}_{h_t,4}\mathbf{w}(t)=\begin{bmatrix}
0 & \eth_{h_t,4}w_2(t) & \cdots & \eth_{h_t,4}w_{160}(t) & 0
\end{bmatrix}^\top.
\end{align*}
The choice of fourth order approximations $\eth_{h_t,4}w_k(t)$ of each derivative $\partial_tw_k(t)$ for $2\leq k\leq 160$, is based on the fact that the synthetic data $\Sigma_{4401}$ recorded in {\tt NLSEqData.csv} were computed using a fourth order Runge-Kutta scheme for time integration.

We will apply algorithm \ref{alg:main_SLMESolver_alg_1} along the lines of corollary \ref{cor:s-dif-model-id} to identify model \eqref{Schrodinger_eq_def_fd}, using a dictionary of $203$ functions $f_j:\mathbb{C}^{161}\to \mathbb{C}^{161}$ defined in terms of a generic $\mathbf{u}=\begin{bmatrix}u_j\end{bmatrix}\in \mathbb{C}^{161}$ by the following expressions
\begin{align*}
f_{1}(\mathbf{u})&=\begin{bmatrix}
0 &
u_{2} &
u_{3} &
\cdots &
u_{160}&
0
\end{bmatrix}^\top,\\
f_{2}(\mathbf{u})&=\begin{bmatrix}
0 &
0 &
u_{2} &
u_{3} &
\cdots &
u_{159}&
0
\end{bmatrix}^\top\\
f_{3}(\mathbf{u})&=\begin{bmatrix}
0 &
u_{3} &
u_{4} &
\cdots &
u_{160} &
0 &
0
\end{bmatrix}^\top\\
f_{4}(\mathbf{u})&=\begin{bmatrix}
0 &
|u_{2}|u_{2} &
|u_{3}|u_{3} &
\cdots &
|u_{160}|u_{160}&
0
\end{bmatrix}^\top,\\
f_{5}(\mathbf{u})&=\begin{bmatrix}
0&
|u_{2}|^2u_{2} &
|u_{3}|^2u_{3} &
\cdots &
|u_{160}|^2u_{160}&
0
\end{bmatrix}^\top,\\
\vdots\\
f_{203}(\mathbf{u})&=\begin{bmatrix}
0&
|u_{2}|^{200}u_{2} &
|u_{3}|^{200}u_{3} &
\cdots &
|u_{160}|^{200}u_{160}&
0
\end{bmatrix}^\top.
\end{align*}

Since the synthetic data $\Sigma_{4401}$ was generated using Runge-Kutta approximation of a second order finite difference discretization of \eqref{Schrodinger_eq_def}, there is some numerical noise induced in the synthetic signal due to floating point errors, in addition, for this experiment some pseudorandom noise with $\mathcal{O}(10^{-6})$ is added to the reference data $\Sigma_{4401}$ to obtain a noisier version $\tilde{\Sigma}_{4401}$ of $\Sigma_{4401}$ that has been recorded in \cite{CodeVides} as {\tt NoisyNLSEqData.csv} for future references. 

Some reference data corresponding to the amplitude $|w|$ of $w$ together with the predictions computed with the model that has been identified using the sparse solver of the SDSI toolset are visualized in figure \ref{fig:NLSEPredictedAmplitudes}.
\begin{figure}[h]
\centering
\includegraphics[scale=.55]{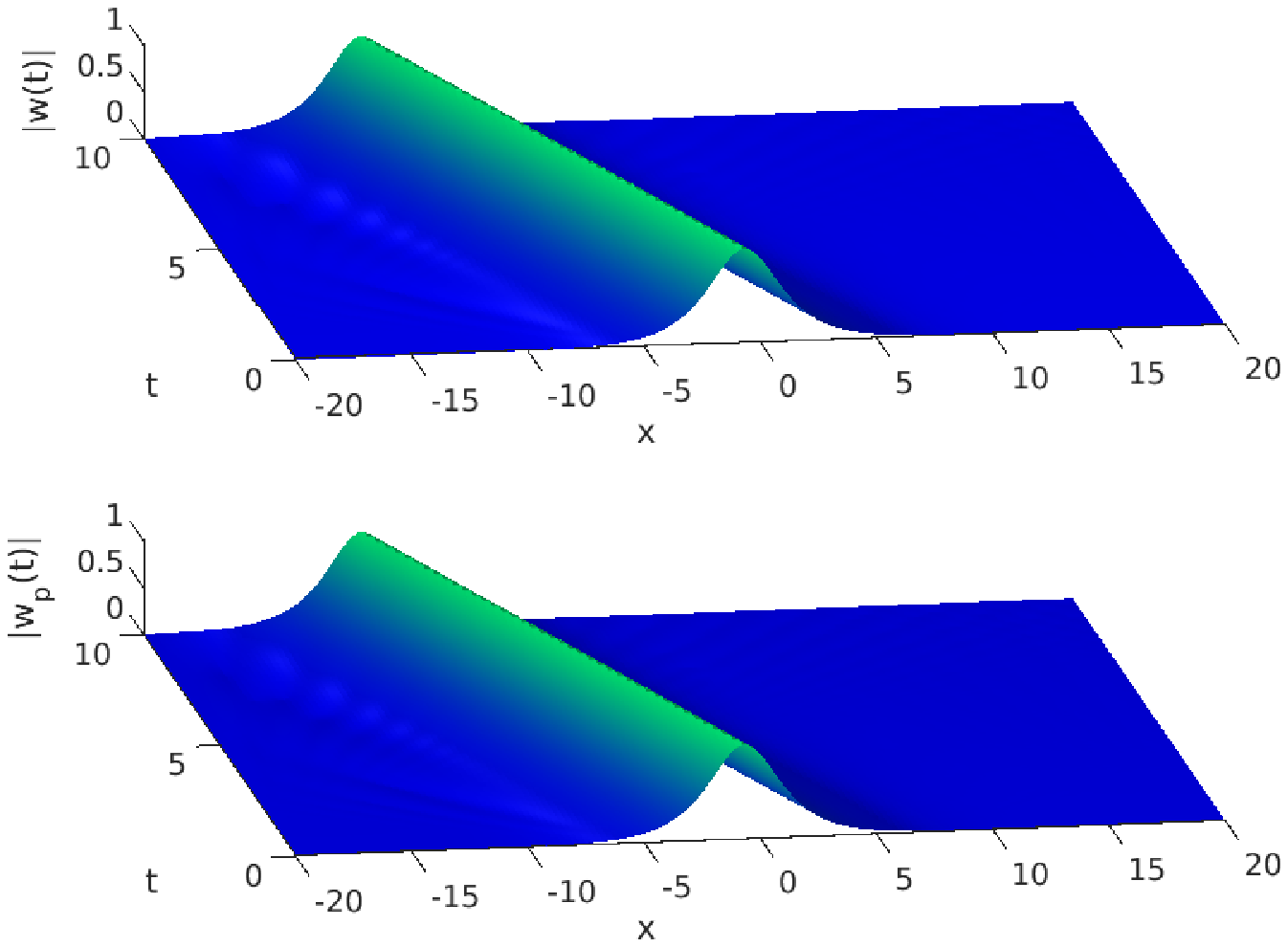}
\caption{Reference data amplitudes $|w(t)|$ (top). Predicted values of the amplitudes $|w_p(t)|$ computed with SDSI model identifier (bottom).}
\label{fig:NLSEPredictedAmplitudes}
\end{figure}

The sparse identification of the signal data $\Sigma_{4401}$ corresponding to the discretization \eqref{Schrodinger_eq_def_fd} of \eqref{Schrodinger_eq_def} was also computed using standard SINDy and Douglas-Rachford sparse solvers as presented and implemented in \cite{BruntonSINDy} and \cite{StructuredDouglasRachfordID}. The absolute prediction errors corresponding to the SDSI, SINDy and Douglas-Rachford solvers in the $\ell_\infty$-norm are shown in figure \ref{fig:NLSEPredictionErrors}.
\begin{figure}[h]
\centering
\includegraphics[scale=.55]{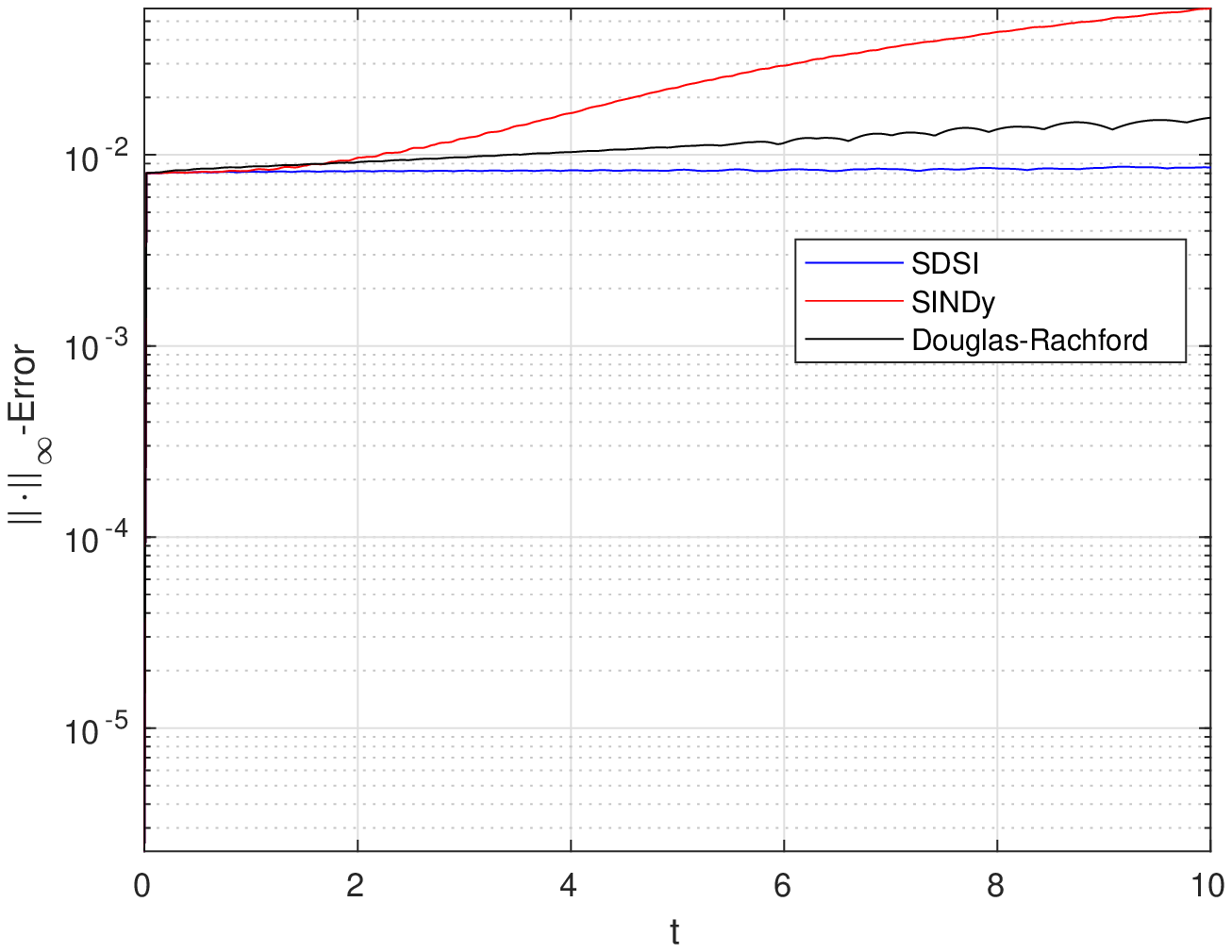}
\caption{Prediction errors in the $\ell_\infty$-norm.}
\label{fig:NLSEPredictionErrors}
\end{figure}
The running times are documented in the table \ref{table:RTtable}.
\begin{center}
 \begin{tabular}{||c | c||} 
 \hline
 Method & Running Time (seconds)\\ [0.5ex] 
 \hline\hline
 SDSI &  0.122988\\ 
 \hline\hline 
 SINDy & 0.258558\\
 \hline 
 \hline\hline 
 Douglas-Rachford & 286.864170\\
 \hline
 \end{tabular}
\captionof{table}{Running time of the sparse identification methods}
\label{table:RTtable}
\end{center}
Let us write $w_{k}$ to denote the $k$-component of the state vector $\mathbf{w}(t)\in \mathbb{C}^{161}$ determined by the expression $w_k=w(-20+(k-1)h_x,th)$ for $h_x=1/4$ and $h=1/10$. The corresponding semi-discrete model will be
\begin{align*}
i\dot{w}_{k}&=-32w_{k}+16w_{k+1}+16w_{k-1}+|w_k|^2w_k, 2\leq k\leq 160
\end{align*}
with $w_1=w_{161}=0$. The semi-discrete models identified by each method based on a training set of $35$ samples $\tilde{\Sigma}_{35}=\{\tilde{\mathbf{w}}(1),\ldots,\tilde{\mathbf{w}}(35)\}\subset \tilde{\Sigma}_{4401}$ for each $2\leq k\leq 160$ are documented in table \ref{table:NLSEIDModels}, for every model in the table $w_1=w_{161}=0$.
\begin{center}
 \begin{tabular}{||c | c||} 
 \hline
 Method & Indentified Model\\ [0.5ex] 
 \hline\hline
 SDSI &  $\!\begin{aligned}
 i\dot{w}_{k}&=(-31.8766 - 0.0276i)w_{k}+(15.9414 + 0.0211i)w_{k+1}\\
 &+(15.9381 + 0.0059i)w_{k-1}+(0.9962 + 0.0008i)|w_k|^2w_k
\end{aligned}$\\ 
 \hline\hline 
 SINDy & $\!\begin{aligned}
 i\dot{w}_{k}&=(-31.8705 - 0.0646i)w_{k}+(15.9408 + 0.0390i)w_{k+1}\\
 &+(15.9330 + 0.0231i)w_{k-1}+(-0.0340 + 0.1333i)|w_k|w_k\\
 &+(2.8996 - 6.5725i)|w_k|^2w_k+(-52.3400 + 160.4634i)|w_k|^3w_k\\
 &+(790.9174 - 2.2437\times 10^3)|w_k|^4w_k+ (25 \:\: \mathrm{more} \:\: \mathrm{terms})
\end{aligned}$\\
 \hline
 \hline\hline 
 Douglas-Rachford & $\!\begin{aligned}
 i\dot{w}_{k}&=(-29.0860 - 0.0599i)w_{k}+(14.5809 + 0.2070i)w_{k+1}\\
 &+(14.5735 - 0.1480i)w_{k-1}+(0.9042 + 0.0012i)|w_k|^2w_k\\
 &+(0.0045 - 0.0002i)|w_k|^3w_k
\end{aligned}$\\
 \hline
 \end{tabular}
\captionof{table}{Model Indentified by each method}
\label{table:NLSEIDModels}
\end{center}

The computational setting used for this experiment is documented in the Matlab program \\
{\tt NLSESpModelID.m} in \cite{CodeVides} that can be used to replicate this experiment.

\subsubsection{Sparse identification of a network of Duffing oscillators with symmetries}
In this section a network of three coupled Duffing oscillators with the following configuration
\begin{align}
&\dot{x}_i=y_i, \:\: i\in \{1,2,3\}\nonumber\\
&\dot{y}_i=\sigma y_j-x_i(\beta+\alpha^2 x_i)+\sum_{ij} \eta_{ij}(x_i-x_j), \:\: i\in \{1,2,3\} \nonumber\\
&x_1(0)=8,x_2(0)=7,x_3(0)=4,\nonumber\\
&y_1(0)=15,y_2(0)=14,y_3(0)=9
\label{eq:DuffingModel}
\end{align}
is identified, for $\alpha=1,\beta=-36,\sigma=0$ and with all the coupling strengths $\eta_{ij}$ equal to $1/5$. The configuration of these experiment is based on the example II.2 considered in \cite{ChaosEqKoopman}, an important part of the motivation for the study of this types of networks of oscillators comes from interesting applications in engineering and biological cybernetics like the ones presented in \cite{GroupsAndNLOscillators}, \cite{SymmetriesMLNerworks} and \cite{GaitsSymmetries}. Since all the coupling strengths $\eta_{ij}$ are equal to $1/5$, as stablished in \cite{ChaosEqKoopman} the system \eqref{eq:DuffingModel} will be $D_3$-equivariant, and for the configuration used for this experiment, the matrix representation of the corresponding group of symmetries $D_3=\langle r,\kappa| r^3=\kappa^2=e,\kappa r\kappa=r^{-1}\rangle$ is determined by the following assignments.
\begin{align*}
r&\mapsto r_\rho =I_2\otimes \begin{bmatrix}
0 & 1 & 0\\
0 & 0 & 1\\
1 & 0 & 0
\end{bmatrix}\\
\kappa&\mapsto \kappa_\rho= I_2\otimes \begin{bmatrix}
1 & 0 & 0\\
0 & 0 & 1\\
0 & 1 & 0
\end{bmatrix}
\end{align*}
The synthetic signal used for sparse identification of \eqref{eq:DuffingModel} was computed with an adaptive fourth order Runge-Kutta scheme. The model was trained using the dictionary
\begin{align*}
f_1(x_1,x_2,x_3,y_1,y_2,y_3)&=(x_1,x_2,x_3,y_1,y_2,y_3)\\
f_2(x_1,x_2,x_3,y_1,y_2,y_3)&=(x_1^2,x_2^2,x_3^2)\\
f_3(x_1,x_2,x_3,y_1,y_2,y_3)&=(x_1^3,x_2^3,x_3^3)\\
\vdots\\
f_9(x_1,x_2,x_3,y_1,y_2,y_3)&=(x_1^9,x_2^9,x_3^9)
\end{align*}
with $20\%$ of the synthetic reference data.

The reference synthetic signal and the corresponding identified signals are illustrated in figure \ref{fig:NLOSignals}
\begin{figure}[h]
\centering
\includegraphics[scale=.6]{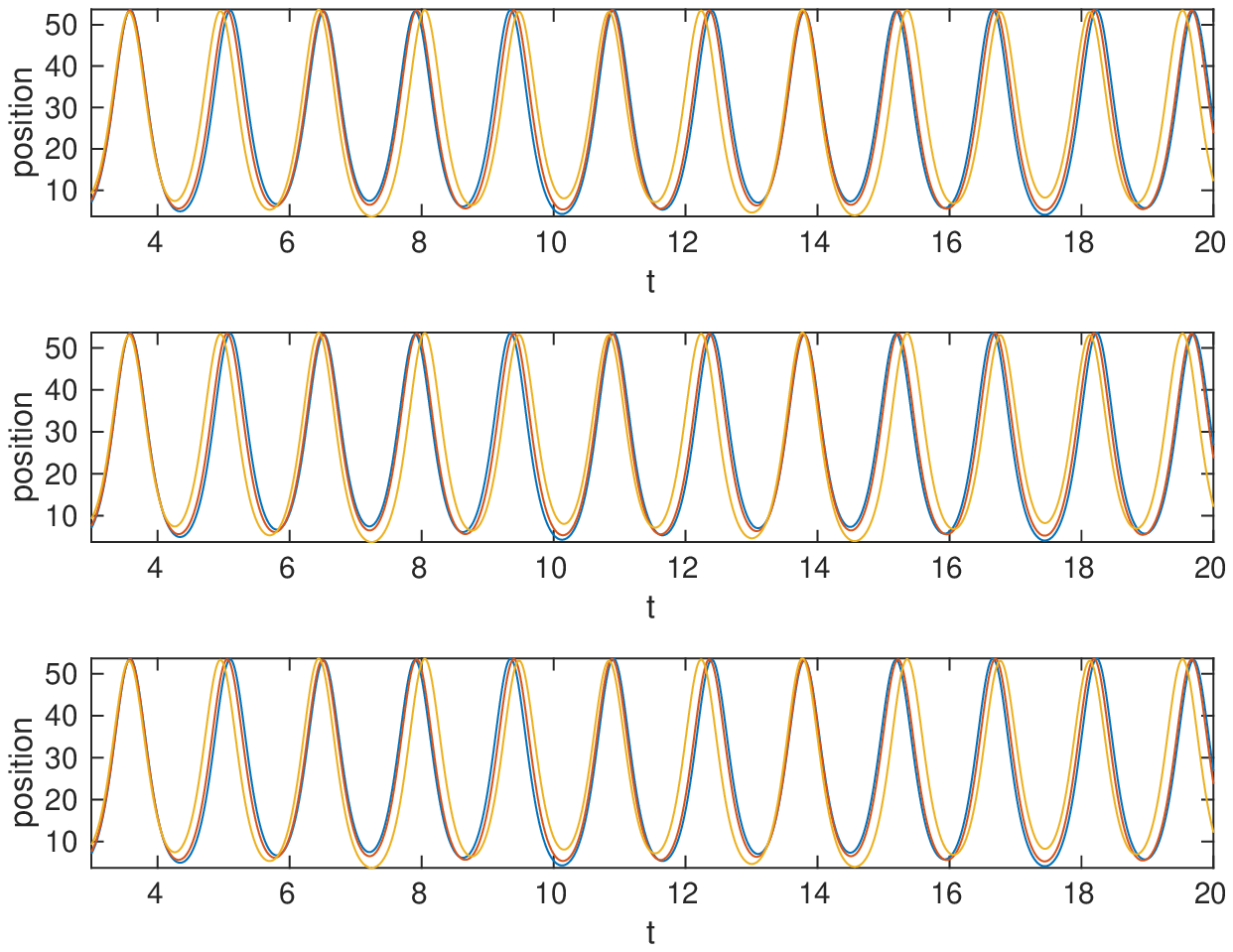}
\caption{Synthetic reference position signals (top). Signals identified using SDSI sparse solver (middle). Signals identified using SINDy sparse solver (bottom).}
\label{fig:NLOSignals}
\end{figure}

The prediction errors $\|\mathbf{x}(t)-\mathbf{x}_p(t)\|$, and the equivariance errors $\|F(r_\rho\mathbf{x}(t))-r_\rho F(\mathbf{x}(t))\|$ and $\|F(\kappa_\rho\mathbf{x}(t))-\kappa_\rho F(\mathbf{x}(t))\|$ of the identified right hand side $F$ for a model $\dot{x}=F(x)$ of the form \eqref{eq:DuffingModel}, for each discrete time index $t$ are plotted in figure \ref{fig:NLOErrors}.
\begin{figure}[h]
\centering
\includegraphics[scale=.6]{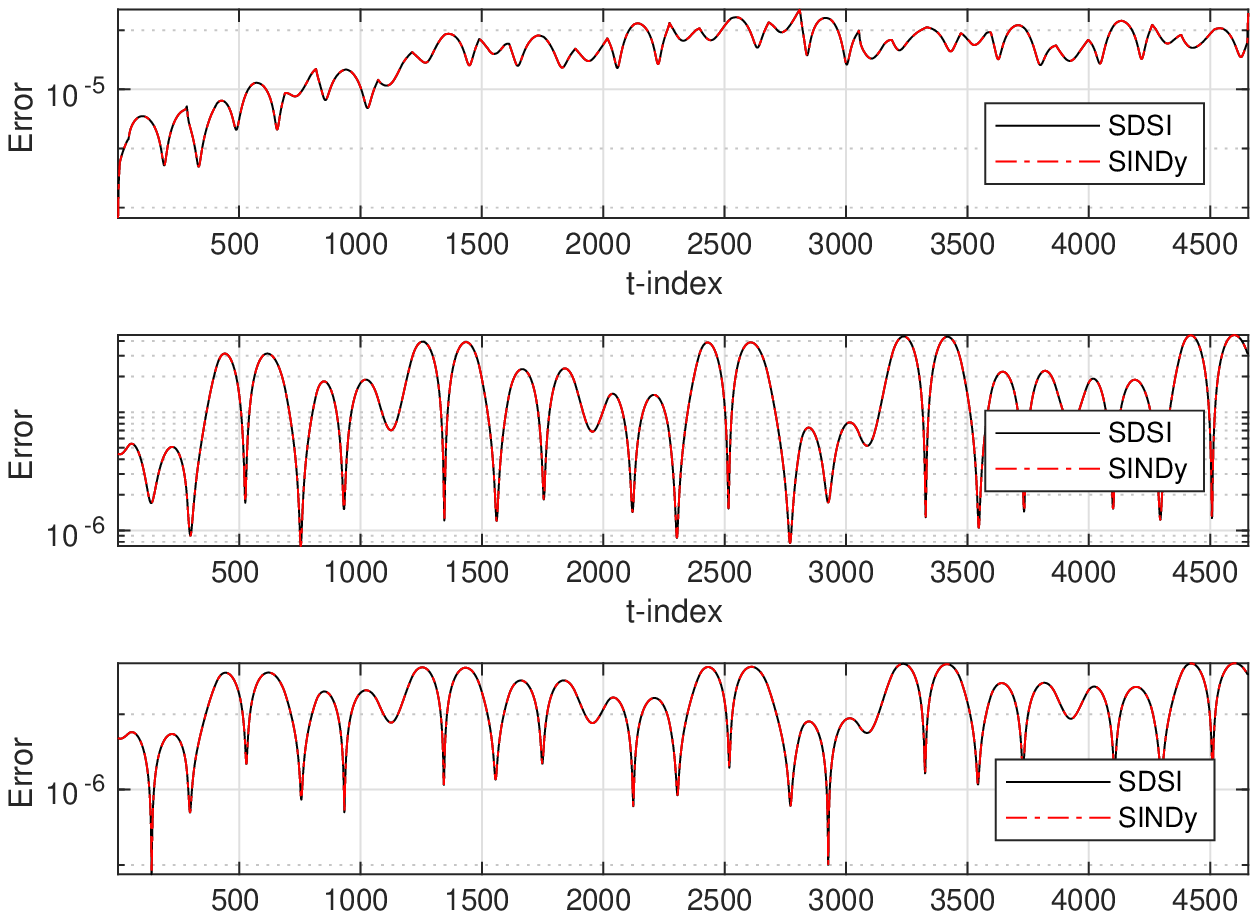}
\caption{Prediction errors (top). Equivariance errors for $r_\rho$ (middle). Equivariance errors for $\kappa_\rho$ (bottom).}
\label{fig:NLOErrors}
\end{figure}

The models identified by each method are documented in tables \ref{table:NLOIDModels-1} and \ref{table:NLOIDModels-2}.

\begin{center}
 \begin{tabular}{||c | c||} 
 \hline
 Method & Indentified Model\\ [0.5ex] 
 \hline\hline
 SDSI &  $\!\begin{aligned}
\dot{x}_1&=0.999999992419860y_1,\\
\dot{x}_2&=0.999999992358200y_2,\\
\dot{x}_3&=0.999999992182884y_3,\\
\dot{y}_1&=36.399999678455494x_1-0.200001381665299x_2-0.199999513388808x_3\\
&-0.999999969696376x_1^2,\\
\dot{y}_2&=-0.199999469429072x_1+36.399998021632314x_2-0.199999719964769x_3\\
&-0.999999971061446x_2^2,\\
\dot{y}_3&=-0.199998908401240x_1-0.200001498693973x_2+36.399999282587004x_3\\
&-0.999999971862063x_3^2,
\end{aligned}$\\ 
 \hline\hline 
 \end{tabular}
\captionof{table}{Model Indentified by SDSI.}
\label{table:NLOIDModels-1}
\end{center}

\begin{center}
 \begin{tabular}{||c | c||} 
 \hline
 Method & Indentified Model\\ [0.5ex] 
 \hline\hline
 SINDy &  $\!\begin{aligned}
\dot{x}_1&=0.999999992419860y_1,\\
 \dot{x}_2&=0.999999992358200y_2, \\
 \dot{x}_3&=0.999999992182883y_3,\\
\dot{y}_1&=36.399999678455423x_1-0.200001381665187x_2-0.199999513388848x_3\\
&-0.999999969696376x_1^2,\\
\dot{y}_2&=-0.199999469429017x_1+36.399998021632094x_2-0.199999719964673x_3\\
&-0.999999971061444x_2^2,\\
\dot{y}_3&=-0.199998908401247x_1-0.200001498693980x_2+36.399999282587018x_3\\
&-0.999999971862063x_3^2,
\end{aligned}$\\
 \hline
 \end{tabular}
\captionof{table}{Model Indentified by SINDy.}
\label{table:NLOIDModels-2}
\end{center}

The running times for the computation of each model are documented in table \ref{table:NLModelRTimes}.
\begin{center}
 \begin{tabular}{||c | c||} 
 \hline
 Method & Running Time (seconds)\\ [0.5ex] 
 \hline\hline
 SDSI &  $0.002407$ \\ 
 \hline\hline 
 SINDy & $0.001376$\\
 \hline
 \end{tabular}
\captionof{table}{Running times for the computation of the nonlinear model identifications.}
\label{table:NLModelRTimes}
\end{center}

Using the same data we can apply algorithm \ref{alg:main_alg_1} to compute a local linear model approximant for the model \ref{eq:DuffingModel}. The identified signals computed using SDSI sparse solver are shown in figure \ref{fig:NLOLTISDSI} and the 
identified signals computed using SINDy sparse solver are shown in figure \ref{fig:NLOLTISINDy}.
\begin{figure}[h]
\centering
\includegraphics[scale=.6]{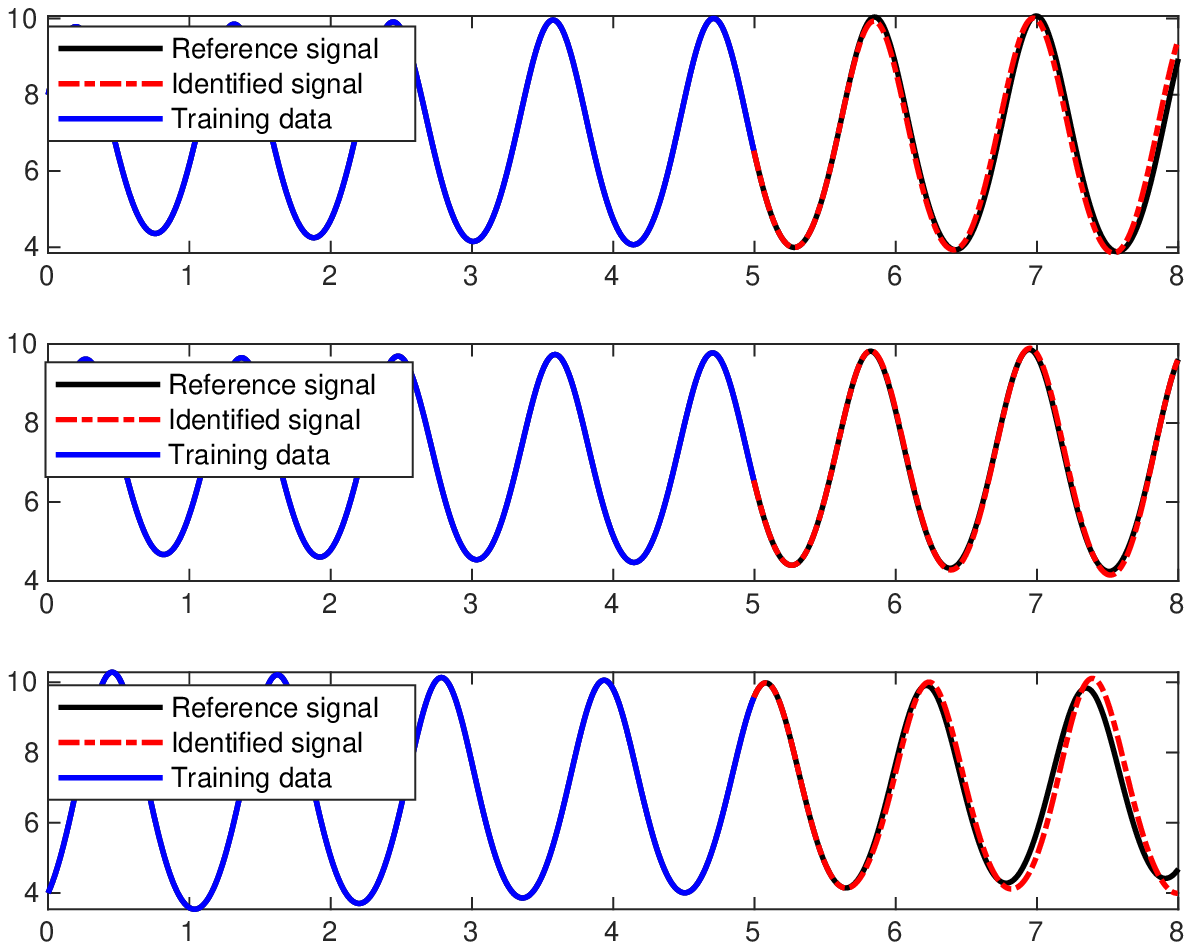}
\caption{Signal identification computed with SDSI for: $x_1$ (top). $x_2$ (middle). $x_3$ (bottom).}
\label{fig:NLOLTISDSI}
\end{figure}

\begin{figure}[h]
\centering
\includegraphics[scale=.6]{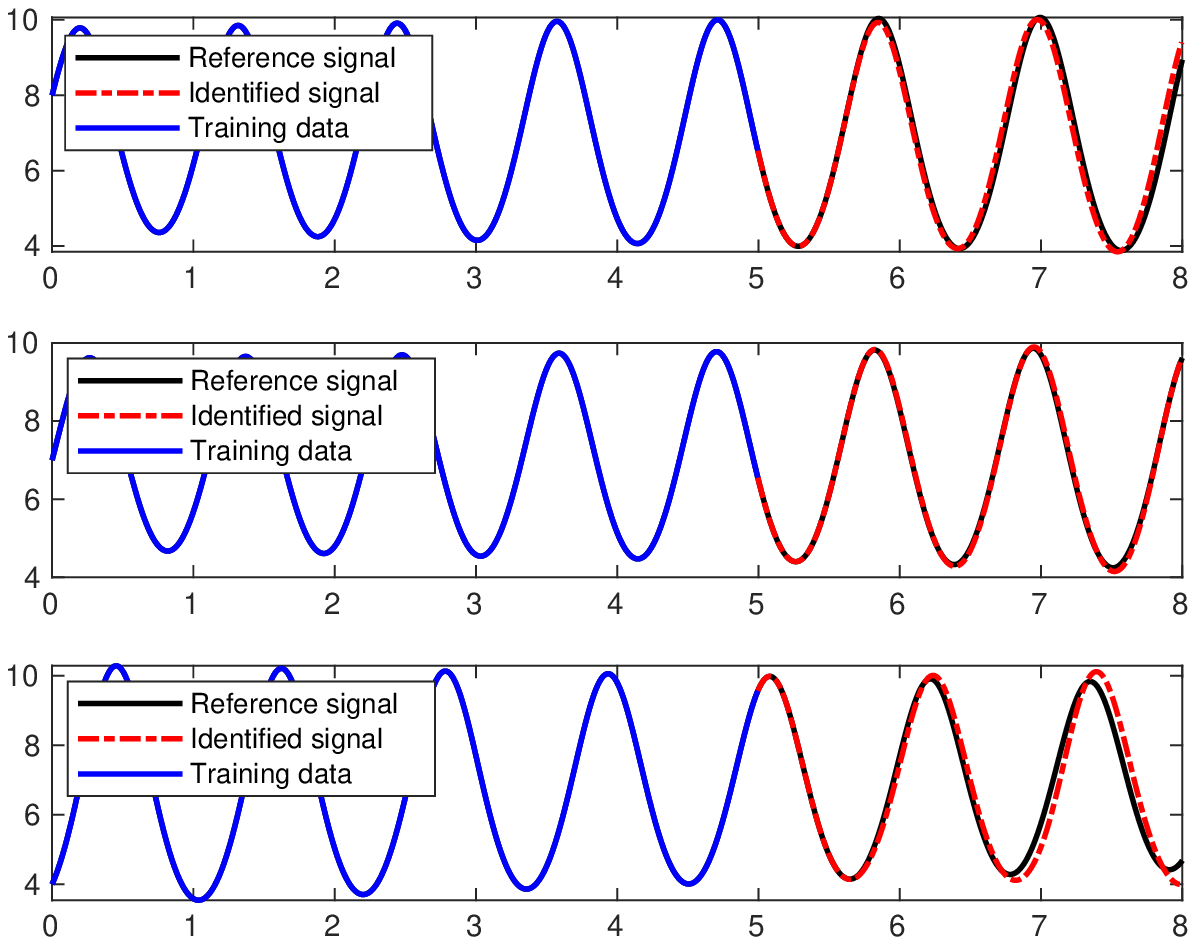}
\caption{Signal identification computed with SINDy for: $x_1$ (top). $x_2$ (middle). $x_3$ (bottom).}
\label{fig:NLOLTISINDy}
\end{figure}

The sparsity patterns of the matrices of parameters identified by each method are shown in figure \ref{fig:NLOMatrixID}.

\begin{figure}[h]
\centering
\includegraphics[scale=.6]{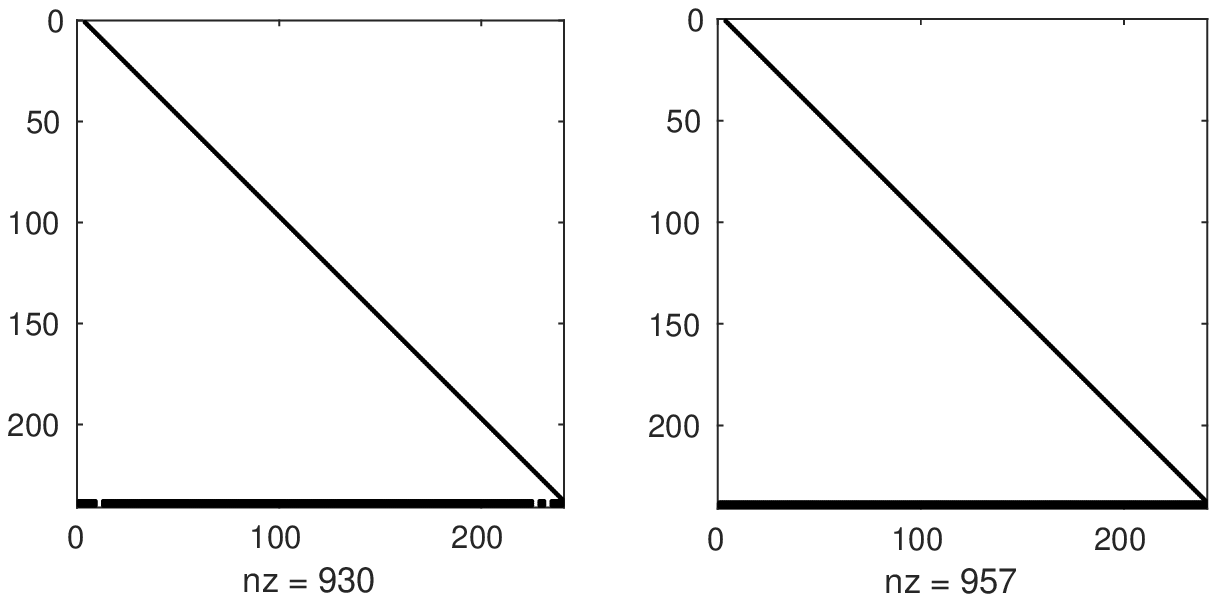}
\caption{Sparsity pattern of the matrix of parameters identified by: SDSI (left). SINDy (right).}
\label{fig:NLOMatrixID}
\end{figure}

Let us set $\hat{r}_\rho=I_L\otimes r_\rho$ and $\hat{\kappa}_\rho=I_L\otimes \kappa_\rho$. The numerical errors corresponding to the symmetry constraints imposed to the matrices of parameters are documented in table \ref{table:LTIErrorstable}.
\begin{center}
 \begin{tabular}{||c | c| c||} 
 \hline
 Method & $\|\hat{r}_\rho A_T-A_T\hat{r}_\rho\|_F$ & $\|\hat{\kappa}_\rho A_T-A_T\hat{\kappa}_\rho\|_F$ \\ [0.5ex] 
 \hline\hline
 SDSI &  $6.311151928321685\times 10^{-16}$ & $0$ \\ 
 \hline\hline 
 SINDy & $9.138315630890346\times 10^{-16}$ & $0$\\
 \hline
 \end{tabular}
\captionof{table}{Numerical errors corresponding to the symmetry preserving constraints.}
\label{table:LTIErrorstable}
\end{center}

The root mean square errors corresponding to the variables considered for the local linear model approximants are documented in table \ref{table:LTIRMSE}, and the running times corresponding to the computation of the local linear model approximants are documented in table \ref{table:LTIRTimes}.

\begin{center}
 \begin{tabular}{||c | c| c||} 
 \hline
 Variable & RMSE (SDSI) & RMSE (SINDy) \\ [0.5ex] 
 \hline\hline
 $x_1$ &  $0.153492169453791$ & $0.154497134958370$ \\ 
 \hline\hline 
 $x_2$ & $0.032257712638561$ & $0.032739961411393$\\
 \hline\hline 
 $x_3$ & $0.227124393816712$ & $0.227994328981342$\\
 \hline
 \end{tabular}
\captionof{table}{Root mean square errors corresponding to each method and variable.}
\label{table:LTIRMSE}
\end{center}

\begin{center}
 \begin{tabular}{||c | c||} 
 \hline
 Method & Running Time (seconds)\\ [0.5ex] 
 \hline\hline
 SDSI &  $0.095987$ \\ 
 \hline\hline 
 SINDy & $0.420293$\\
 \hline
 \end{tabular}
\captionof{table}{Running times for the computation of the local linear model approximant.}
\label{table:LTIRTimes}
\end{center}

The computational setting used for the experiments performed in this section is documented in the Matlab programs {\tt NLONetworkID.m} and {\tt DuffingLTIModelID.m} in \cite{CodeVides} that can be used to replicate these experiments.

\subsubsection{Sparse identification of a time series model for vortex shedding processes}
Let us consider a Navier-Stokes nonlinear model of the form
\begin{align}
\frac{\partial {u}}{\partial{t}} +u\cdot\nabla u-\nu \Delta u+\nabla p&=0\nonumber \\
 \nabla\cdot u&=0 
\label{Navier_Stokes_eq_def}
\end{align}
under suitable geometric configuration, initial and boundary conditions that lead to vortex shedding. For this experiment, the sparse model identification process is based on the synthetic data corresponding to a vortex shedding process recorded in the file {\tt GFUdata.csv}, that is included in \cite{CodeVides} as the compressed file {\tt GFUdata.zip}, the synthetic data were generated using the program {\tt navierstokes-tcsi.py} included in \cite{CodeVides} that is based on the Netgen/Python program {\tt navierstokes.py} developed by J. Sch{\"o}berl as part of the work initiated with \cite{NETGEN}.

The time series model identification based on the data $\Sigma_{502}=\{u_t\}_{t=1}^{502}\subset \mathbb{R}^{53770}$ recorded in {\tt GFUdata.csv} is performed with a computational implementation of algorithm \ref{alg:main_alg_1} using the Matlab program {\tt NSIdentifier.m} which estimates the number $\mathrm{deg}_{9\times 10^{-7},\{I_{53770}\}}(\Sigma_{400})$ obtaining the value $\mathrm{deg}_{9\times 10^{-7},\{I_{53770}\}}(\Sigma_{400})=1$, for this experiment we have that $\mathrm{rk}_{\delta}(\mathscr{H}_1(\Sigma_{400}))=354$, and the models are trained with the subsample $\Sigma_{354}=\{u_t\}_{t=1}^{354}\subset \Sigma_{400}\subset \Sigma_{502}$.

Since $d=1$, the models can be computed using the reduced form 
\begin{align*}
u_{k+1}=\mathscr{H}_1(\Sigma_{353})A_{353}^k\hat{e}_{1,353}, \:\: k\geq 1
\end{align*}
with $\Sigma_{353}=\{u_{t}\}_{t=1}^{353}$, and only a matrix of parameters $A_{353}$ is left to be identified. The matrix of parameters will be identified using the sparse solver {\tt SpSolver.m} from the SDSI toolset based on algorithm \ref{alg:main_SLMESolver_alg_1} and the sparse solver {\tt SINDy.m}  based on SINDy sparse least squares solver algorithm introduced in \cite{BruntonSINDy}.

The sparsity patterns of the matrices of parameters identified by the Matlab programs {\tt SpSolver.m} and {\tt SINDy.m} are shown in figure \ref{fig:NSSPPatterns}.
\begin{figure}[h]
\centering
\includegraphics[scale=.6]{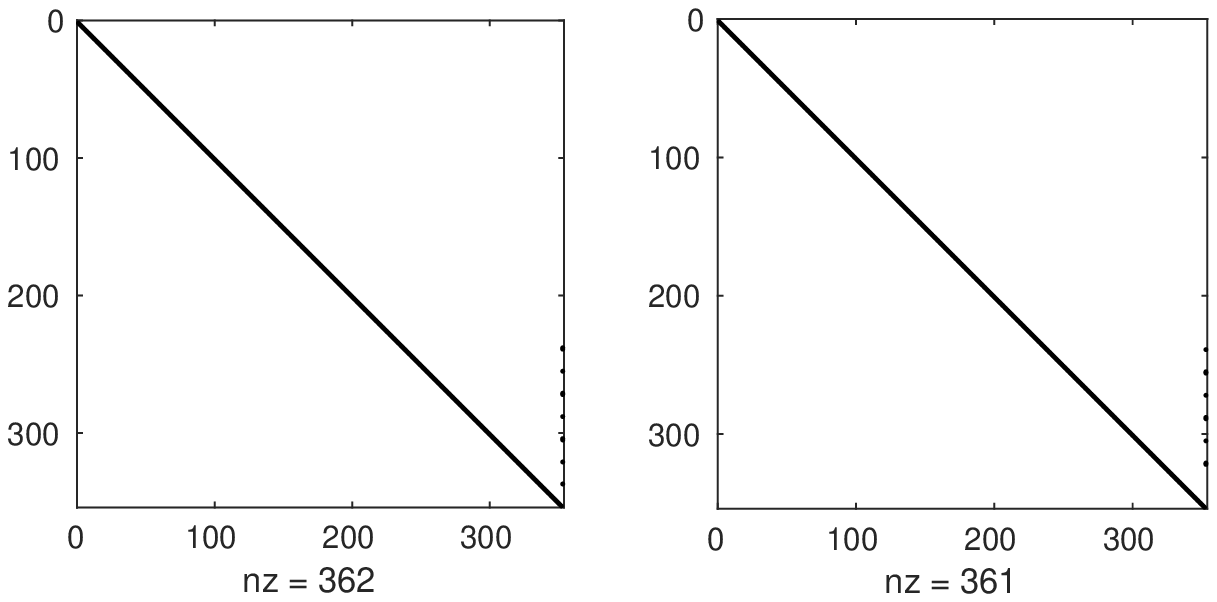}
\caption{Sparsity patterns of matrices of parametres identified by: SINDy solver (left) and SDSI solver (right).}
\label{fig:NSSPPatterns}
\end{figure}

The reference state $u_{502}$ and the corresponding states predicted by SINDy and SDSI algorithms are shown in figure \ref{fig:NSStates}.
\begin{figure}[h]
\centering
\includegraphics[scale=.3]{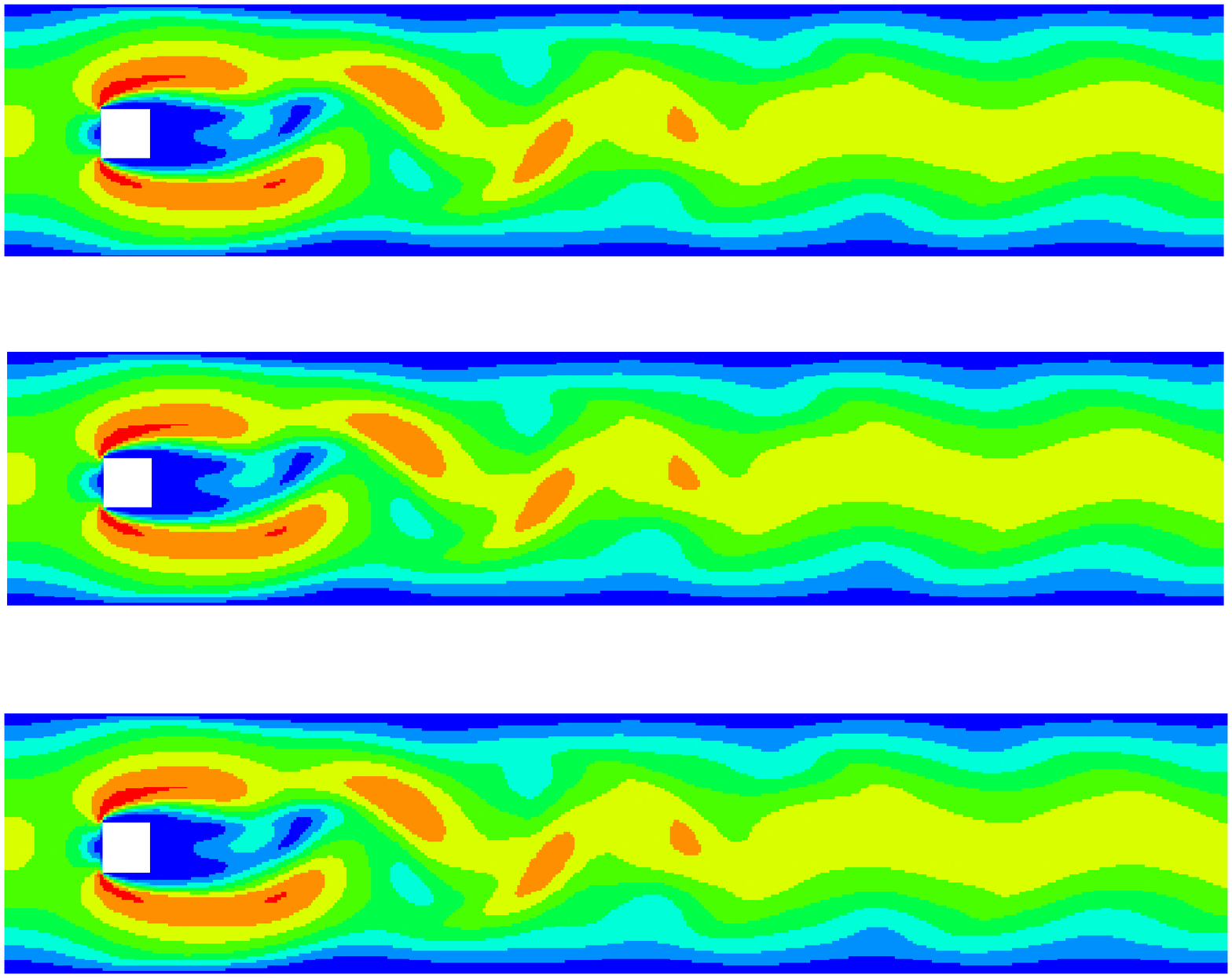}
\caption{Reference $u_{502}$ state (top). Predicted $\hat{u}_{502}$ state with SDSI (middle). Predicted $\hat{u}_{502}$ state with SINDy (bottom).}
\label{fig:NSStates}
\end{figure}

The prediction error estimates $\|\mathscr{H}_1(\Sigma_{502})-\mathscr{H}_1(\hat{\Sigma}_{502})\|_F$ for the predicted states $\hat{\Sigma}_{502}$ computed with each method are documented in table \ref{table:NSLTIErrors}, and the running times corresponding to the computation of the sparse representations of the matrices of parameters are documented in table \ref{table:NSLTIRTimes}.

\pagebreak

\begin{center}
 \begin{tabular}{||c | c||} 
 \hline
 Method & $\|\mathscr{H}_1(\Sigma_{502})-\mathscr{H}_1(\hat{\Sigma}_{502})\|_F$ \\ [0.5ex] 
 \hline\hline
 SDSI &  $0.0057$ \\ 
 \hline\hline 
 SINDy & $0.0072$\\
 \hline
 \end{tabular}
\captionof{table}{Prediction error estimates corresponding to each method.}
\label{table:NSLTIErrors}
\end{center}

\begin{center}
 \begin{tabular}{||c | c||} 
 \hline
 Method & Running Time (seconds)\\ [0.5ex] 
 \hline\hline
 SDSI &  $1.371872$ \\ 
 \hline\hline 
 SINDy & $6.017688$\\
 \hline
 \end{tabular}
\captionof{table}{Running times for the computation of the local linear model approximant.}
\label{table:NSLTIRTimes}
\end{center}

The computational setting used for this experiment is documented in the Matlab script \\
{\tt NSIdentifier.m} in \cite{CodeVides}. There are also Julia and Python versions of this script included in \cite{CodeVides}, but only the Matlab version was documented in this article as it is the fastest version of the algorithm \ref{alg:main_alg_1}, probably due in part to the outstanding performance of the least squares solver implemented in Matlab via the {\tt "$\backslash$"} operator.

\subsubsection{Sparse identification of a weather forecasting model}
For this example we will use a subcollection of the weather time series dataset recorded by the Max Planck Institute for Biogeochemistry, that was collected between 2009 and 2016 and prepared by François Chollet for \cite{DLBook}. From the column {\tt T (degC)} of the corresponding table, a subsample $\Sigma_{8412}\subset\mathbb{R}$ uniformly sampled in time with $8412$ elements has been extracted and recorded as {\tt TemperatureData.csv} for future references as part of \cite{CodeVides}. 

Using $50\%$ of the data in $\Sigma_{8412}$ we can compute three weather forecasting models of the form
\begin{align*}
T_{k+1}&=c_{0}+c_{1}T_{k}+c_{2}T_{k-1}+\cdots+c_{L}T_{k-L+1}\\
\approx& c_0+\mathscr{P}_LA_{T}^k\begin{bmatrix}
T_{1} \\ T_{2} \\ \vdots \\ T_{L}\end{bmatrix}.
\end{align*}

The first model will be computed using the function {\tt AutoReg} from the module {\tt statsmodels} from Python, the second model will be computed applying algorithm \ref{alg:main_alg_1} with the sparse solver {\tt SINDy.m}, and the third model will be computed applying algorithm \ref{alg:main_alg_1} with the sparse solver {\tt SpSolver.m}.

The reference signal, the predicted signal and the model parameters $c_{k}$ computed with {\tt AutoReg} are shown in figure \ref{fig:AutoRegTSModelData}.
\begin{figure}[h]
\centering
\includegraphics[scale=.45]{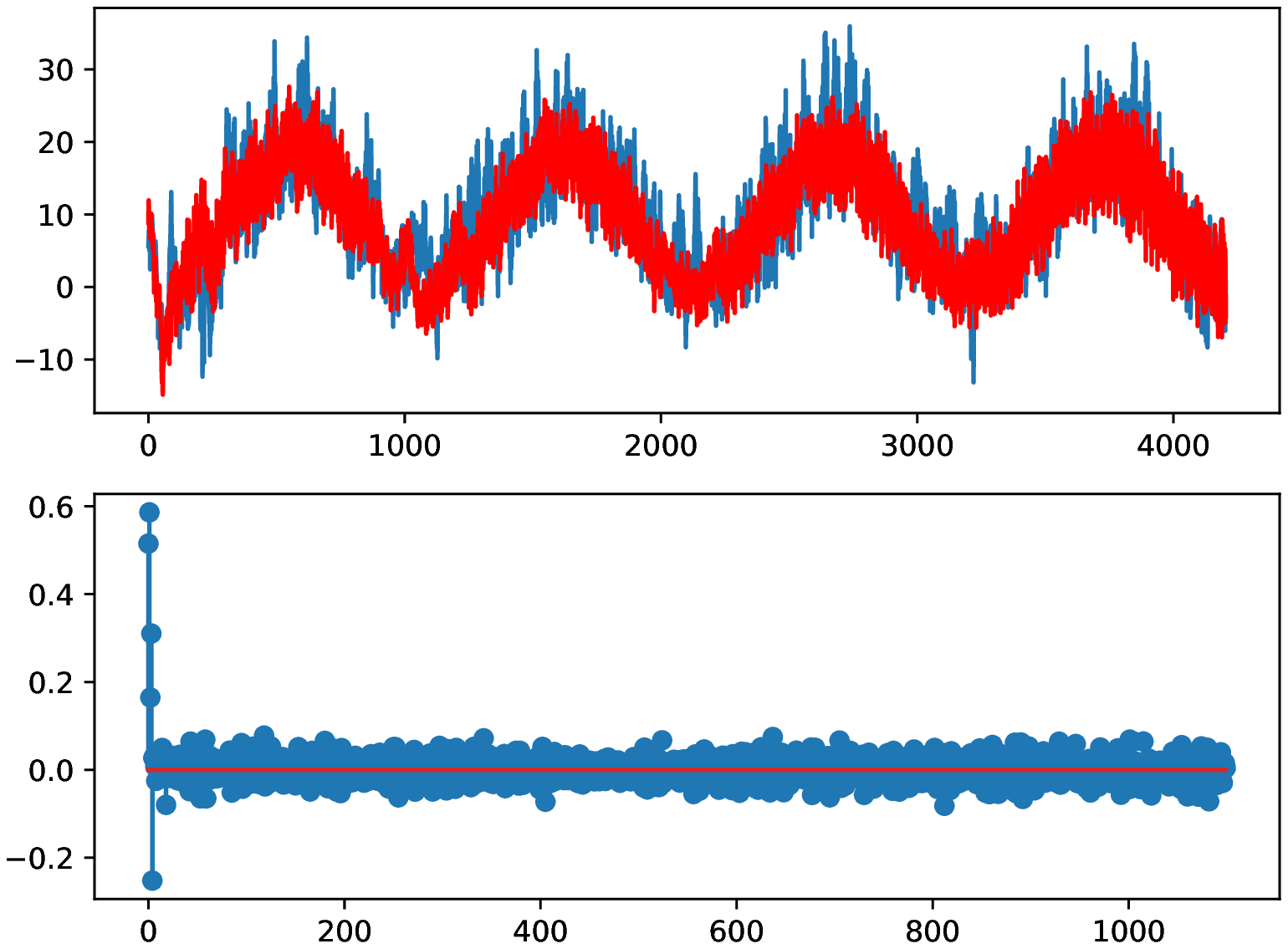}
\caption{Reference signal (blue line) and model prediction (red line) (top). Model parameters computed using {\tt AutoReg}. (bottom).}
\label{fig:AutoRegTSModelData}
\end{figure}
Since the model computed with {\tt AutoReg} has $1100$ nonzero coefficients, the recurrence relation corresponding to the model computed with {\tt AutoReg} will not be written explicitly. 

The reference signal, the predicted signal and the model parameters $c_{k}$ computed with algorithm \ref{alg:main_alg_1} using {\tt SINDy.m} are shown in figure \ref{fig:SINDyTSModelData}.
\begin{figure}[h]
\centering
\includegraphics[scale=.55]{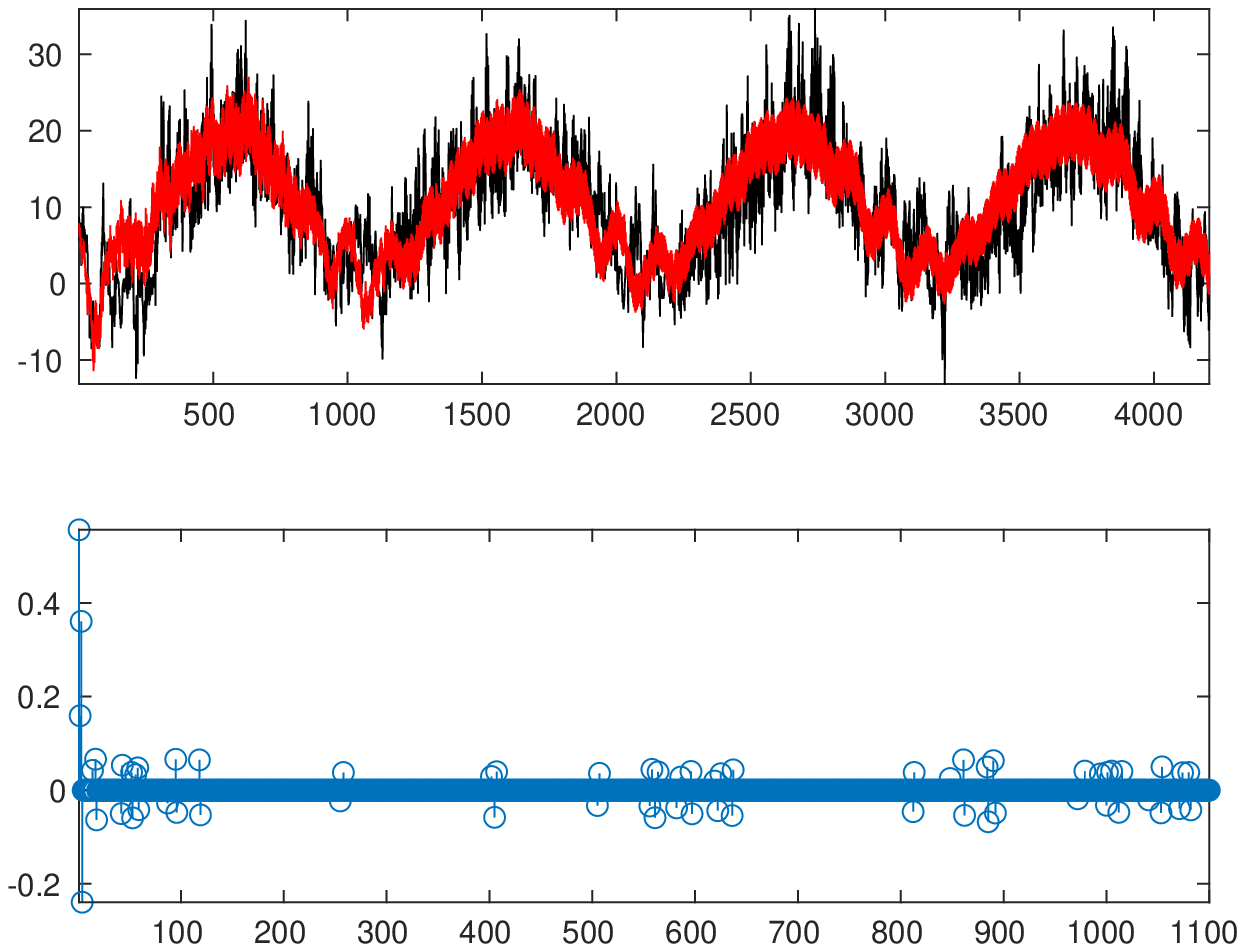}
\caption{Reference signal (black line) and model prediction (red line) (top). Model parameters computed using {\tt SINDy.m} (bottom).}
\label{fig:SINDyTSModelData}
\end{figure}

The reference signal, the predicted signal and the model parameters $c_{k}$ computed with algorithm \ref{alg:main_alg_1} using {\tt SpSolver.m} are shown in figure \ref{fig:SDSITSModelData}.
\begin{figure}[h]
\centering
\includegraphics[scale=.55]{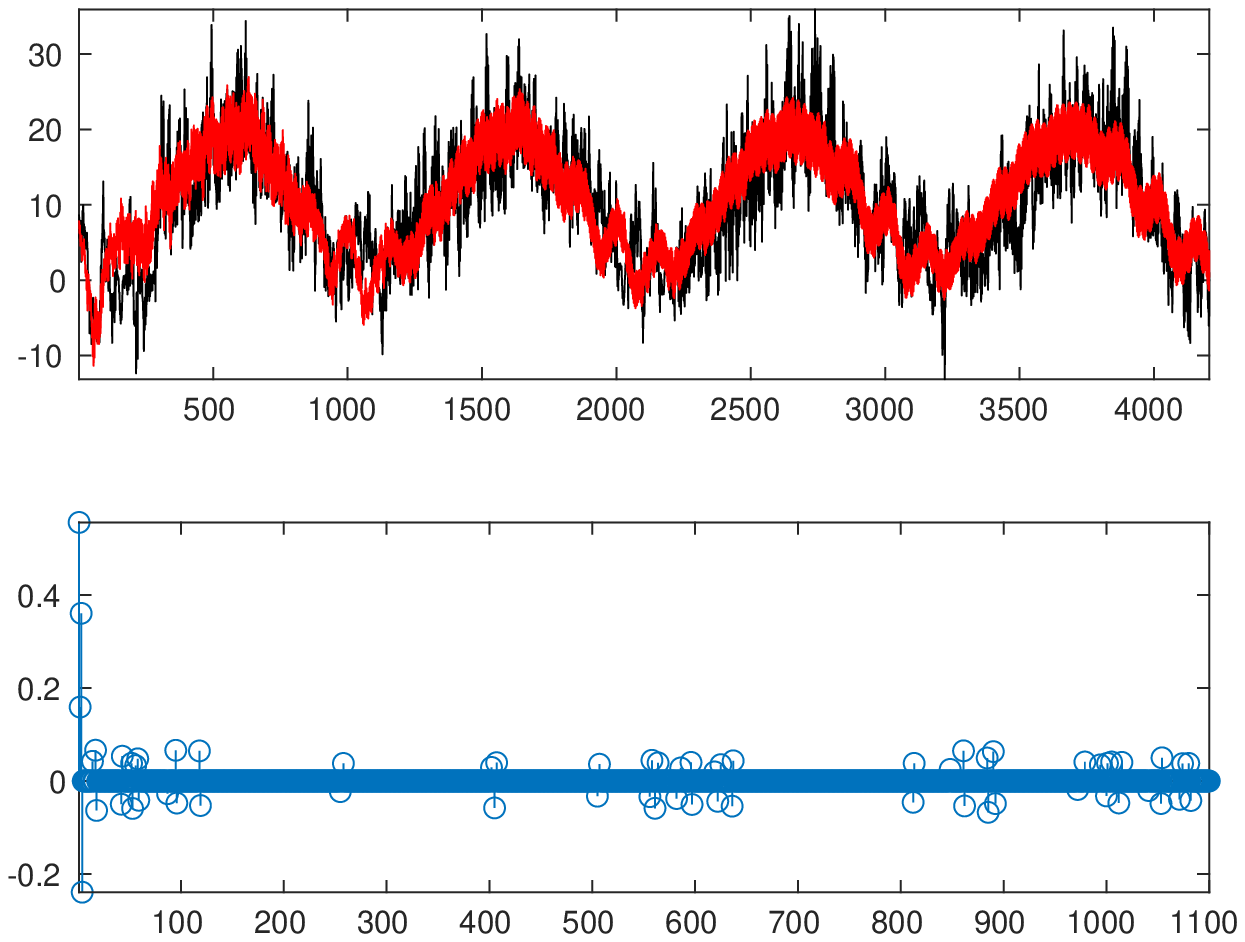}
\caption{Reference signal (black line) and model prediction (red line) (top). Model parameters computed using {\tt SpSolver.m} (bottom).}
\label{fig:SDSITSModelData}
\end{figure}

Although the models computed with {\tt SpSolver.m} and {\tt SINDy.m}, both have only 63 nonzero parameters $c_k$, we will not write the recurrence relations corresponding to these models explicitly either. 

The root mean square error estimates for each model are documented in table \ref{table:WeatherTSErrors}, and the running times corresponding to the computation of the model parameters are documented in table \ref{table:WeatherTSRTimes}.

\pagebreak

\begin{center}
 \begin{tabular}{||c | c||} 
 \hline
 Method & RMSE\\ [0.5ex] 
 \hline\hline
 AutoReg &  $6.222$ \\ 
 \hline\hline 
 SINDy & $5.8945$\\
  \hline\hline 
 SDSI &  $5.8945$ \\ 
 \hline
 \end{tabular}
\captionof{table}{Prediction error estimates corresponding to each method.}
\label{table:WeatherTSErrors}
\end{center}

\begin{center}
 \begin{tabular}{||c | c||} 
 \hline
 Method & Running Time (seconds)\\ [0.5ex] 
 \hline\hline
 AutoReg &  $1.0370969772338867$ \\ 
 \hline\hline
 SINDy &  $0.896653$ \\ 
 \hline\hline 
 SDSI & $0.562108$\\
 \hline
 \end{tabular}
\captionof{table}{Running times for the computation of the local linear model approximant.}
\label{table:WeatherTSRTimes}
\end{center}

The computational setting used for the experiments performed in this section is documented in the Matlab program {\tt SpTSPredictor.m} and the Python program {\tt TSModel.py} in \cite{CodeVides}, these programs can be used to replicate these experiments.

\section{Conclusion}

The results in \S\ref{section:linear-solvers} and \S\ref{section_SDSI} in the form of algorithms like the ones described in \S\ref{section_algorithms}, can be effectively used for the sparse identification of dynamical models that can be used to compute data-driven predictive numerical simulations. 

One of the main advantages of the low-rank approximation approach presented in this document, concerns the parameter estimation for linear and nonlinear models where numerical or measurement noises could affect the estimates significantly, an example of this phenomenon is documented as part of the numerical experiment \ref{exa:example-NLSE}.

\section{Future Directions}

As a consequence of theorem \ref{thm:main-thm}, it is possible to approach system identification problems as approximate matrix completion problems via low-rank matrix approximation, the corresponding connections with local low-rank matrix approximation in the sense of \cite{LLORMA} will be studied as part of the future directions of this research project.

As observed by Koch in \cite{DDTravelingWaves}, when applying sparse system identification techniques like the ones presented in this document, arriving at a high-quality model will be strongly related with the choice of functions in the method’s library, which may require expert knowledge or tuning, and this consideration aligns perfectly with the philosophy behind the development of the system identification technology presented as part of the results reported in this article.

Computational implementations of the the sparse solvers and identification algorithms presented in this document to the computation of sparse signal models and signal compressions, using rectangular sparse submatrices of wavelet matrices, will also be the subject of future communications.

The connections of the results in \S\ref{section_SDSI} to the solution of problems related to controllability and realizability of finite-state systems in classical and quantum information and automata theory in the sense of  \cite{finite_state_systems,finite_quantum_control_systems,finite_state_machine_approximation,BROCKETT20081}, will be further explored. 

Further applications of sparse signal model identification schemes to industrial automation and building information modeling (BMI) technologies, will be further explored.

\section*{Data Availability}

The programs and data that support the findings of this study are openly available in the SDSI repository, reference number \cite{CodeVides}.

\section*{Conflicts of Interest}
The author declares that he has no conflicts of interest.

\section*{Acknowledgment}

The structure preserving matrix computations needed to implement the algorithms in \S\ref{section_algorithms}, were performed with  {\sc Matlab}  R2021a (9.10.0.1602886) 64-bit (glnxa64), Python 3.8.5, Julia 1.6.0 and Netgen/NGSolve 6.2, with the support and computational resources of the Scientific Computing Innovation Center ({\bf CICC-UNAH}) of the National Autonomous University of Honduras.

I am grateful with Terry Loring, Stan Steinberg, Concepci\'on Ferrufino, Marc Rieffel and Moody Chu for interesting conversations, that have been very helpful for the preparation of this document.

\bibliographystyle{plain}
\bibliography{SPSystemIDFVides}

\end{document}